\documentclass[a4paper,12pt]{article}
\usepackage[utf8]{inputenc}
\usepackage[english]{babel}
\usepackage[ruled,section]{algorithm}
\usepackage{makeidx}
\usepackage[font=small]{caption}
\usepackage{amsmath}
\usepackage{amsthm}
\usepackage{amsfonts}
\usepackage{amssymb}
\usepackage{mathrsfs}
\usepackage{graphicx}
\usepackage[margin=25mm]{geometry}
\usepackage{enumerate}
\usepackage{indentfirst}
\usepackage[none]{hyphenat}
\usepackage{color}
\usepackage{authblk}
\makeindex
\newcommand{\real}{\mathbb{R}}
\newcommand{\n}{\mathbb{N}}

\newcommand{\rN}{ {\mathbb{R}^N} }

\newcommand{\rmd}{\mathrm{d}}
\newcommand{\F}{\mathcal{F}}
\newcommand{\C}{\mathcal{C}}
\newcommand{\Ps}{\mathcal{P}}
\newcommand{\M}{\mathcal{M}}
\newcommand{\Sl}{\mathcal{S}}

\numberwithin{equation}{section}
\newtheorem{theorem}{Theorem}[section]
\newtheorem{lem}[theorem]{Lemma}
\newtheorem{prop}[theorem]{Proposition}
\newtheorem{corol}[theorem]{Corollary}

\theoremstyle{definition}

\newtheorem{rmk}[theorem]{Remark}

\newtheorem{defin}[theorem]{Definition}

\begin{document}

\title{\bf\Large A dynamical system approach to a class of radial weighted fully nonlinear equations}
\author[1]{Liliane Maia\footnote{l.a.maia@mat.unb.br}}
\author[2]{Gabrielle Nornberg\footnote{gabrielle@icmc.usp.br}}
\author[3]{Filomena Pacella\footnote{pacella@mat.uniroma1.it}}
\affil[1]{\small Universidade de Brasília, Brazil}
\affil[2]{\small Instituto de Ciências Matemáticas e de Computação, Universidade de São Paulo, Brazil}
\affil[3]{\small Sapienza Università di Roma, Italy}

{\date{\today}}

\maketitle

{\small\noindent{\bf{Abstract.}} 
In this paper we study existence, nonexistence and classification of radial positive solutions of some weighted fully nonlinear equations involving Pucci extremal operators. 
Our results are entirely based on the analysis of the dynamics induced by an autonomous quadratic system which is obtained after a suitable transformation.
This method allows to treat both regular and singular solutions in a unified way, without using energy arguments.
In particular we recover known results on regular solutions for the fully nonlinear non weighted problem by alternative proofs. 
We also slightly improve the classification of the solutions for the extremal operator $\M^-$.}
\medskip

{\small\noindent{\bf{Keywords.}} {Fully nonlinear equations; critical exponents; regular and singular solutions; dynamical system.}

\medskip

{\small\noindent{\bf{MSC2020.}} {35J15, 35J60, 35B09, 34A34.}

\section{Introduction and main results}\label{Introduction}

In this paper we study positive radial solutions of the following class of fully nonlinear elliptic equations
\begin{align}\label{P} 
\mathcal{M}_{\lambda,\Lambda}^\pm (D^2 u)+|x|^a u^p =0  , \quad
u> 0 \;\; \textrm{ in }  \;\Omega,
\end{align}
where $a> -1$, $p>1$, and $\mathcal{M}^\pm$ are the Pucci's extremal operators which play an essential role in stochastic control theory and mean field games. Here, $0<\lambda\leq\Lambda$ are the ellipticity constants, see Section \ref{section 2} and \cite{CafCab} for their properties.
The set $\Omega\in \rN$, $N\geq 3$, is a radial domain such as $\rN$, a ball $B_R$ of radius $R>0$ centered at the origin, the exterior of $B_R$, or an annulus. 
However, most of the time $\Omega$ will be either the whole space or a ball.

We deal with both regular and singular solutions $u$ of \eqref{P} which are $C^2$ for $r>0$. In the singular case $\Omega$ will be either $\rN\setminus\{0\}$ or $B_R\setminus\{0\}$, and we assume the condition
\begin{align}\label{H singular} \textstyle{\lim_{r\to 0} \,u(r) = +\infty , \;\;\; r=|x|.}
\end{align}
Finally, whenever $\Omega$ has a boundary, we prescribe the Dirichlet condition 
\begin{align}\label{H Dirichlet}
u=0 \textrm{ on } \partial\Omega , \quad \textrm{ or } \quad u=0 \textrm{ on $\partial\Omega\setminus\{0\}$\, under } \eqref{H singular}.
\end{align}

Let us recall some previous results when $a=0$. 
A general existence result in bounded domains $\Omega$ (not necessarily radial) was obtained in \cite{BQnonproper} under the condition 
\begin{align}\label{BQ numbers}
\textstyle{1<p\leq \frac{\tilde{N}_+}{\tilde{N}_+-2} \;\;\textrm{ for } \M^+_{\lambda,\Lambda}\, , \qquad 1<p\leq \frac{\tilde{N}_-}{\tilde{N}_--2} \;\;\textrm{ for } \M^-_{\lambda,\Lambda}}
\end{align}
if $\tilde{N}_+>2$, where $\tilde{N}_\pm$ are the so called dimensional-like numbers
\begin{align}\label{def dimensional-like}
\textstyle{\tilde{N}_+=\frac{\lambda}{\Lambda}(N -1)+1, \qquad \tilde{N}_-=\frac{\Lambda}{\lambda}(N -1)+1.}
\end{align}

The intervals in \eqref{BQ numbers} describe the optimal range for existence of  supersolutions to $\M^\pm$, as shown in \cite{CutriLeoni}. Some extensions and related results can be found in \cite{BApisa, BAScpam}.

When the Pucci's operators reduce to the Laplacian (i.e.\ for $\lambda=\Lambda$), both exponents \eqref{BQ numbers} are equal to $\frac{N}{N-2}$ which is known as Serrin exponent. They do not provide optimal bounds in terms of solutions of \eqref{P}, as it is clear for instance by considering the semilinear case. 

Nevertheless, as far as the radial setting is concerned, critical exponents which represent the threshold for the existence of solutions to \eqref{P} can be defined. They were introduced for $a=0$ by Felmer and Quaas in \cite{FQaihp} in order to study existence and classification of radial positive solutions in $\rN$.
These are also the watershed for existence and nonexistence of positive solutions in the ball. Note that every positive solution in the ball when $a=0$ is radial, by \cite{BdaLioSym}, while this is not true in general for $a\neq 0$, even in the semilinear case.

When $\lambda=\Lambda$  the corresponding critical exponents are the same, both in radial and nonradial settings; see \cite{CGS} for $a=0$, and \cite{GGN} for $a\neq 0$. The identification of critical exponents in the nonradial case for fully nonlinear operators is still open.

Let us recall some preliminary definitions.

\begin{defin}\label{def regular}
Let $u$ be a radial solution of \eqref{P} for $\Omega=\rN$. Set $r=|x|$ and $\alpha=\frac{2+a}{p-1}$. Then $u$ is said to be:\vspace{-0.2cm}
\begin{enumerate}[(i)]
\item \textit{fast decaying} if there exists $c>0$ such that $\lim_{r\to\infty} r^{\tilde{N}-2} u(r) = c$, where $\tilde{N}$ is either $\tilde{N}_+$ if the operator is $\mathcal{M}^+$ or $\tilde{N}_-$ for $\mathcal{M}^-$, see \eqref{def dimensional-like};\vspace{-0.1cm}
		
\item \textit{slow decaying} if there exists $c>0$ such that $\lim_{r\to\infty} r^\alpha u(r) = c$;\vspace{-0.1cm}
		
\item \textit{pseudo-slow decaying} if there exist constants $0<c_1<c_2$ such that \vspace{-0.25cm}
\begin{center}
$c_1 = \liminf_{r\to\infty} r^\alpha u(r) < \limsup_{r\to\infty} r^\alpha u(r) = c_2$.
\end{center}
\end{enumerate}
\end{defin} 

The definitions (i) and (ii) are classical from the theory of Lane-Emden equations. In turn (iii) was introduced in \cite{FQaihp} and is peculiar of the fully nonlinear case. It corresponds to solutions oscillating at $+\infty$ by changing concavity infinitely many times. 

\begin{theorem}[Theorems 1.1 and 1.2 in \cite{FQaihp}]\label{FQ Th introd}
Assume $a=0$, $\tilde{N}_+>2$, and $\lambda<\Lambda$. Then there exist critical exponents $p^*_+$, $p^*_-$ satisfying the bounds \vspace{-0.1cm}
\begin{center}
$\max\left\{\frac{\tilde{N}_+}{\tilde{N}_+-2}, \frac{N+2}{N-2}\right\}< p^*_+< \frac{\tilde{N}_++2}{\tilde{N}_+-2}$\quad and \quad $\frac{\tilde{N}_-+2}{\tilde{N}_--2}
< p^*_-< \frac{N+2}{N-2}$,
\end{center}
such that the following holds for $\Omega=\rN$:
\vspace{-0.2cm}
	\begin{enumerate}[(i)]
	\item if $p\in (1,p^*_{\pm})$ there is no nontrivial radial solution of \eqref{P};\vspace{-0.1cm}
		
	\item if $p=p^*_{\pm}$ there exists a unique fast decaying radial solution to \eqref{P};\vspace{-0.1cm}
		
	\item If $p>p^*_{\pm}$ there exists a unique  radial solution of \eqref{P}, which is either slow decaying or pseudo-slow decaying.\vspace{-0.2cm}
	\end{enumerate}
In (i) and (ii) uniqueness is meant up to scaling.
\end{theorem}

In addition, in the case of $\mathcal{M}^+$ the authors made precise the range of the exponent $p$ for which pseudo-slow decaying solutions exist.

The existence of a critical exponent unveils an important feature of the Pucci's operators. It reflects some intrinsic properties of these operators and induces concentration phenomena besides of energy invariance, see \cite{BGLPcv}, as it happens in the classical semilinear case.

The proof of Theorem \ref{FQ Th introd} in \cite{FQaihp} is involved. It is a combination of the Emden-Fowler phase plane analysis and the Coffman Kolodner technique. The latter consists in differentiating the solution with respect to the exponent $p$, and then studying a related nonhomogeneous differential equation, from which they derive the behavior of the solutions for $p$ on both right and left hand sides of $p^*_\pm$, as well as the uniqueness of the exponent $p$ for which a fast decaying solution exists.

In this paper we study both regular and singular solutions of the more general weighted problem \eqref{P}. For the regular ones we prove results similar to those of Theorem \ref{FQ Th introd} but with different tools entirely based on a dynamical system approach.

Our solutions are understood in the classical sense out of $0$ and they are of class $C^1$ up to $0$, since $a>-1$, see Proposition \ref{prop decreasing concave}.
Before stating our results it is useful to fix some notations for relevant exponents, depending also on the number $a > -1$ which characterizes the weight in \eqref{P}:
\begin{align}\label{critical exponents a}
{p^{p,a}_\pm=\frac{\tilde{N}_\pm+2a+2}{\tilde{N}_\pm-2},	\quad p^{s,a}_\pm=\frac{\tilde{N}_\pm+a}{\tilde{N}_\pm - 2},	\quad p_\Delta^a=\frac{N+2+2a}{N-2},	\quad \alpha =\frac{2+a}{p-1}.}
\end{align}

\begin{theorem}[$\mathcal{M}^+$ regular solutions]\label{Th M+}
Assume $\tilde{N}_+>2$, and $\lambda<\Lambda$. Then there exists a critical exponent $p^*_{a+}$ such that \vspace{-0.4cm}
\begin{align}\label{critical ordering M+}
\max\{p^{s,a}_+,p^a_\Delta\} <p^*_{a+}< p^{p,a}_+,
\end{align}\vspace{-0.1cm}
and the following assertions hold:
	\begin{enumerate}[(i)]
		\item if $p\in (1,p^*_{a+})$ there is no nontrivial radial solution of \eqref{P} in the whole $\rN$, while for any $R>0$ there exists a unique radial solution in the ball $B_R$;
		
		\item if $p=p^*_{a+}$ there exists a unique fast decaying radial solution of \eqref{P} in $\rN$;
		
		\item if $p\in (p_{a+}^*,p^{p,a}_+]$ there is a unique pseudo-slow decaying radial solution to \eqref{P} in $\rN$;
		
		\item if $p>p^{p,a}_+$ there exists a unique slow decaying radial solution of \eqref{P} in $\rN$;
		
		\item if $p>p_{a+}^*$ there is no nontrivial solution to \eqref{P}, \eqref{H Dirichlet} when $\Omega$ is a ball.
	\end{enumerate}\vspace{-0.05cm}
In (ii)--(iv) uniqueness is meant up to scaling.
\end{theorem}


\begin{theorem}[$\mathcal{M}^-$ regular solutions]\label{Th M-}
If $\lambda<\Lambda$, then there exists a critical exponent $p^*_{a-}$ satisfying
\begin{align}\label{critical ordering M-}
\textstyle{p^{p,a}_-< p^*_{a-}< p^a_\Delta,}
\end{align}\vspace{-0.1cm}
and there exists $\varepsilon>0$ such that:
\begin{enumerate}[(i)]
	\item if $p\in (1,p^*_{a-})$ there is no nontrivial radial solution of \eqref{P} in the whole $\rN$, while for any $R>0$ there exists a unique radial solution of the Dirichlet problem \eqref{P}, \eqref{H Dirichlet} in $B_R$;

	\item if $p=p^*_{a-}$ there exists a unique fast decaying radial solution of \eqref{P} in $\rN$;
		
	\item if $p\in (p_{a-}^*,p^a_\Delta -\varepsilon]$ there is a unique pseudo-slow or slow decaying radial solution of \eqref{P} in $\rN$;\vspace{-0.2cm}
		
	\item if $p>p^a_\Delta-\varepsilon$ there exists a unique slow decaying radial solution of \eqref{P} in $\rN$;
	
	\item if $p>p_{a-}^*$ there is no nontrivial solution to \eqref{P}, \eqref{H Dirichlet} when $\Omega$ is a ball.
\end{enumerate}\vspace{-0.05cm}
In (ii)--(iv) uniqueness is meant up to scaling.
\end{theorem}

In the $\mathcal{M}^-$ case our result slightly improves the corresponding one of \cite{FQaihp}, for $a=0$, by showing that for $p$ near $p^a_\Delta$ only a slow decaying solution exists; cf.\ point (iv) of Theorem \ref{Th M-}. 

The proofs of the previous theorems rely entirely on a careful analysis of an autonomous quadratic dynamical system that we obtain after a suitable transformation, see Section \ref{section 2}.
It was used in \cite{BV} to study the classical semilinear Lane-Emden system. 
Once the correspondence between the radial solution of \eqref{P} and the orbits of the dynamical systems \eqref{DS+} and \eqref{DS-} is made (see Section \ref{section classification}), all existence and classification results of Theorems \ref{Th M+} and \ref{Th M-} are derived by studying the stationary points and the flow lines of these systems. In particular, the uniqueness of the critical exponent and the behavior of the solutions, by varying the exponent $p$, are obtained as a direct consequence of the properties of the vector fields which define the dynamical systems.

Note that these systems are derived from the Pucci fully nonlinear equations and are piecewise $C^1$. This, in particular, allows the presence of several periodic orbits which produce regular and singular solutions with different features like pseudo-slow decay or pseudo--blowing up behavior at infinity or at the origin.

One reason why our approach is quite simple is that the most relevant sets which determine the flow generated by \eqref{DS+} and \eqref{DS-} are just straight lines; see Figures \ref{Fig flow},\ref{Fig flow2}.
Moreover, the presence of the weight $|x|^a$ in \eqref{P} does not produce additional difficulties, while it could be complicated via the method of \cite{FQaihp}. 
We stress that the usual change of variables which transforms Henon problems into non weighted ones  for $\lambda=\Lambda$ (see for instance \cite{BV, DjairoCM, GGN}) does not seem to work for Pucci's operators when $\lambda\neq \Lambda$ in order to achieve classification results.
Finally we point out that our proofs do not involve any energy function.

On the other hand, by the same analysis of the dynamics induced by \eqref{DS+} and \eqref{DS-} we also get the classification of singular solutions of \eqref{P} in a punctured ball or in $\rN\setminus\{0\}$.
Before stating the results we present some definitions. 

\begin{defin}\label{def singular}
Let $u$ be a radial solution of \eqref{P} and \eqref{H singular}, with either $\Omega=\rN\setminus\{0\}$ or $\Omega=B_R\setminus\{0\}$ for some $R>0$. Then the singular solution $u$ is said to be:
\vspace{-0.1cm}
\begin{enumerate}[(i)]
	\item \textit{$(\tilde{N}-2)$--blowing} up if there exists $c>0$ such that $\lim_{r\to 0} r^{\tilde{N}-2} u(r) = c$, where $\tilde{N}$ is either $\tilde{N}_+$ if the operator is $\mathcal{M}^+$ or $\tilde{N}_-$ for $\mathcal{M}^-$, see \eqref{def dimensional-like};\vspace{-0.1cm}
		
	\item \textit{$\alpha$--blowing up} if there exists $c>0$ such that $\lim_{r\to 0} r^\alpha u(r) = c$, with $\alpha$ as in \eqref{critical exponents a};\vspace{-0.1cm}
		
	\item \textit{pseudo--blowing up} if there exist constants $0<c_1<c_2$ such that \vspace{-0.25cm}
		\begin{center}
			$c_1 = \liminf_{r\to 0} r^\alpha u(r) < \limsup_{r\to 0} r^\alpha u(r) = c_2$.
		\end{center}		
\end{enumerate}
\end{defin} 

We highlight that Definition \ref{def singular} (iii) corresponds to a type of solutions which change concavity infinitely many times in a neighborhood of zero. The existence of such a type of singular solutions was already detected for a more general class of uniformly elliptic equations, for values of the exponent $p$ close to the critical one, see \cite[Section 6]{FQind}.

\begin{rmk}\label{remark trivial singular}
For all $p>p^{s,a}_+$ in the case of $\M^+$, resp.\ $p>p^{s,a}_-$ for $\M^-$, the function $u_p(r)=C_p\,r^{-\alpha}$, $C_p$ as in \eqref{omega M0}, is a singular solution of \eqref{P} in $\rN\setminus\{0\}$. We call it \textit{trivial singular solution}.
\end{rmk}

\begin{theorem}[$\M^+$ singular solutions]\label{Th M+ singular} 
Assuming $\tilde{N}_+>2$ and $\lambda<\Lambda$, for \eqref{P}--\eqref{H singular} it holds:\vspace{-0.1cm}
\begin{enumerate}[(i)]
	\item for any $p\leq p^{s,a}_+$ there is no singular radial solution in $\rN\setminus\{0\}$, while for each $R>0$ there are infinitely many $(\tilde{N}_+-2)$--blowing up radial solutions of \eqref{P}--\eqref{H Dirichlet} in $B_R\setminus\{0\}$;
		
	\item  if $p^{s,a}_+ \leq  p_\Delta^a$ then for any $p\in (p^{s,a}_+,p^a_\Delta ]$ there is a unique $\alpha$--blowing up radial solution in $\rN\setminus\{0\}$ with fast decay at $+\infty$. Also, for any $R>0$ there are infinitely many $\alpha$--blowing up radial solutions of \eqref{P}--\eqref{H Dirichlet} in $B_R\setminus\{0\}$;
		
	\item for each $p\in (\,p_\Delta^a+,p^*_{a+})$ there exists a unique singular radial solution in $\rN\setminus\{0\}$ with fast decay at $+\infty$. Moreover, for any $R>0$ there exist infinitely many singular radial solutions of the Dirichlet problem \eqref{P}--\eqref{H Dirichlet} in $B_R\setminus\{0\}$; 
	
\item if $p=p^*_{a+}$ there exist infinitely many pseudo--blowing up radial solutions in $\rN\setminus\{0\}$ with pseudo-slow decay at $+\infty$, and infinitely many $\alpha$--blowing up in $\rN\setminus\{0\}$  with pseudo-slow decay at $+\infty$. Also, there is no singular radial solution of  \eqref{P}--\eqref{H Dirichlet} in $B_R\setminus\{0\}$;
	
	\item if $p\in (p^*_{a+},p^{p,a}_+)$ there are infinitely many $\alpha$--blowing up radial solutions in $\rN\setminus\{0\}$ with pseudo-slow decay at $+\infty$, and there is a pseudo--blowing up radial solution with pseudo-slow decay at $+\infty$. Further, there is no singular radial solution of  \eqref{P}--\eqref{H Dirichlet} in $B_R\setminus\{0\}$;
		
	\item if $p\in [\, p^{p,a}_+,+\infty)$ there are no nontrivial singular radial solutions, cf.\ Remark \ref{remark trivial singular}.
	\end{enumerate}
Here, uniqueness in $\rN\setminus\{0\}$ is meant up to scaling.
\end{theorem}

\begin{theorem}[$\M^-$ singular solutions]\label{Th singular- Introd} If $\lambda<\Lambda$, for the problem \eqref{P}--\eqref{H singular} we have:
\begin{enumerate}[(i)]
	\item if $p\leq p^{s,a}_-$ there is no singular radial solution in the whole $\rN\setminus\{0\}$, while for any $R>0$ there are infinitely many $(\tilde{N}_- -2)$--blowing up radial solutions of \eqref{P}--\eqref{H Dirichlet} in $B_R\setminus\{0\}$;
		
	\item for each $p\in (p^{s,a}_-,p^{p,a}_-)$ there exists a unique $\alpha$--blowing up radial solution in $\rN\setminus\{0\}$ with fast decay at $+\infty$. Further, for any $R>0$ there exist infinitely many $\alpha$--blowing up radial solutions of the Dirichlet problem \eqref{P}--\eqref{H Dirichlet} in $B_R\setminus\{0\}$;
		
	\item for any $p\in (p^{p,a}_-,p^*_{a-})$ there are infinitely many pseudo--blowing up radial solutions in $\rN\setminus\{0\}$. Among them there is a unique fast decaying, a pseudo-slow decaying, and infinitely many with slow decay at $+\infty$. Moreover, for each $R>0$ there exist infinitely many pseudo--blowing up radial solutions of \eqref{P}--\eqref{H Dirichlet} in $B_R\setminus\{0\}$;
	
	\item if $p=p^*_{a-}$ there exist infinitely many pseudo--blowing up radial solutions in $\rN\setminus\{0\}$. Among them there are infinitely many with slow-decay at $+\infty$, and infinitely many pseudo-slow decaying at $+\infty$. Further, there is no singular radial solution of  \eqref{P}--\eqref{H Dirichlet} in $B_R\setminus\{0\}$;
		
	\item there exists $\varepsilon>0$ such that for $p \in [\, p^a_\Delta-\varepsilon ,+\infty)$ no nontrivial singular radial solution exists.
	\end{enumerate}
Here, uniqueness in $\rN\setminus\{0\}$ is meant up to scaling.
\end{theorem}

Our results on singular solutions are obtained by complementing the analysis of the flow lines of the dynamical systems \eqref{DS+} and \eqref{DS-}. To the best of our knowledge they are the first global classification results on singular solutions found  for this class of fully nonlinear equations. In \cite[Remark 3.2]{FQaihp} it is pointed out that, in the case $a=0$, periodic orbits of the Endem-Fowler system would produce singular solutions, while in \cite[Theorem 6.3]{FQind} the existence of singular solutions is proved near the critical exponent.

For the critical exponents $p^*_{a\pm}$, our dynamical systems \eqref{DS+} and \eqref{DS-} furnish infinitely many periodic orbits. On the other hand, for $p^{p,a}_\pm$ infinitely many periodic orbits appear which do not correspond to $C^2$ solutions for $r>0$, see Remark~\ref{remark singular periodic}.
For $p\in (p^*_{a-}, p_\Delta^a-\varepsilon)$ the existence of singular solutions cannot be deduced directly from the dynamical system approach.

Let us underline the fact that obtaining periodic orbits is in general a very difficult task in the theory of dynamical systems.
Even in the very particular case of a polynomial autonomous system this question is not completely understood, see \cite{ChiconeTian, Hale}.

Finally, as a byproduct of the study of regular radial solutions of \eqref{P}, either in $\rN$ or in a ball, we easily get the range of the exponents $p$ for which a positive radial solution of the Dirichlet problem in the exterior of a ball does not exist. Indeed, we get the following result. 

\begin{theorem}\label{Th exterior ball+-F Introd}
Let $p>1$. Then there are no radial solutions of
\begin{align}\label{ProbExterior} 	\mathcal{M}^\pm (D^2 u)+|x|^a u^p =0  , \quad	u> 0 \;\; \textrm{ in }  \;\rN\setminus B_R ,\quad	u= 0 \;\; \textrm{ on }  \;\partial B_R \end{align}
if $p\leq p^*_{a\pm}$\,  for each $R>0$.	
\end{theorem}

In the case of $a=0$, Theorem \ref{Th exterior ball+-F Introd} has been recently proved in \cite{GILexterior2019} with different arguments which rely both on the study of the second order ODE and on the analysis of the Emden-Fowler system. 
Their work presents a complete picture of existence and nonexistence of solutions for distinct intervals for the values of the parameter $p$.
However, through our arguments we get their nonexistence result by a considerably shorter proof.
Indeed, we will see in Sections \ref{section Pucci+} and \ref{section Pucci-} that the result of Theorem \ref{Th exterior ball+-F Introd} becomes a straightforward consequence of the characterization of the critical exponents $p^*_{a\pm}$ in terms of the associated quadratic system we consider.
Let us point out that in \cite{GILexterior2019} also the existence and classification of the solutions of \eqref{ProbExterior} are provided when $a=0$. Alternatively, this could be done through our methods. Since this is not the main goal of our research we just refer to Section \ref{section Pucci+} for further comments.

To conclude, we stress that another advantage of our approach is that it treats in a unified way several kind of solutions of \eqref{P}. We refer the reader to Figures \ref{Fig p<ps}--\ref{Fig Psingular M-} where, for a given value of the exponent $p$, all the orbits of the system corresponding to different type of solutions of \eqref{P} are displayed simultaneously.

\smallskip

The paper is organized as follows. In Section \ref{section 2} we write down the quadratic system associated to the problem \eqref{P} and study its intrinsic flow properties. In Section~\ref{section classification} we classify the different solutions of \eqref{P} in terms of orbits of the corresponding dynamical systems. Finally, Sections \ref{section Pucci+} and \ref{section Pucci-} are devoted to the proofs of the main results for the Pucci $\mathcal{M}^+$ and $\mathcal{M}^-$ operators, respectively. In the Appendix we provide some details about the stationary points of the dynamical systems, for the reader convenience.

\section{The associated dynamical system}\label{section 2}

In this section we define some new variables which allow to transform the radial fully nonlinear equations into a quadratic dynamical system.
 
We start by recalling that the Pucci's extremal operators $\mathcal{M}^\pm_{\lambda,\Lambda}$, for $0<\lambda\leq \Lambda$, are defined as
$$
\textstyle{\mathcal{M}^+_{\lambda,\Lambda}(X):=\sup_{\lambda I\leq A\leq \Lambda I} \mathrm{tr} (AX)\,,\quad \mathcal{M}^-_{\lambda,\Lambda}(X):=\inf_{\lambda I\leq A\leq \Lambda I} \mathrm{tr} (AX),}
$$
where $A,X$ are $N\times N$ symmetric matrices, and $I$ is the identity matrix.
Equivalently, if we denote by $\{e_i\}_{1\leq i \leq N}$ the eigenvalues of $X$, we can define the Pucci's operators as
\begin{align}\label{def Pucci}
\textstyle{\mathcal{M}_{\lambda,\Lambda}^+(X)=\Lambda \sum_{e_i>0} e_i +\lambda \sum_{e_i<0} e_i, \;\;\; \mathcal{M}_{\lambda,\Lambda}^-(X)=\lambda \sum_{e_i>0} e_i +\Lambda \sum_{e_i<0} e_i}. 
\end{align}
From now on we will drop writing the parameters $\lambda,\Lambda$ in the notations for the Pucci's operators.

In the case when $u$ is a radial function, with an abuse of notation we set $u(|x|)=u (r)$ for $r=|x|$. If in addition $u$ is $C^2$, then the eigenvalues of the Hessian matrix $D^2 u$ are $u^{\prime\prime}$ which is simple, and $\frac{u^\prime (r)}{r}$ with multiplicity $N-1$.  
We then define the Lipschitz functions
\begin{align}\label{m,M+}
m_+(s)=
\begin{cases}
\lambda s\; \textrm{ if } s\leq 0 \\
\Lambda s\; \textrm{ if } s> 0
\end{cases}\;
\textrm{and}\quad
M_+(s)=
\begin{cases}
s/\lambda\; \textrm{ if } s\leq 0 \\
s/ \Lambda\; \textrm{ if } s> 0;
\end{cases}
\end{align}
\vspace{-0.6cm}
\begin{align}\label{m,M-}
m_-(s)=
\begin{cases}
\Lambda s\; \textrm{ if } s\leq 0 \\
\lambda s\; \textrm{ if } s> 0
\end{cases}\;
\textrm{and}\quad
M_-(s)=
\begin{cases}
s/\Lambda\; \textrm{ if } s\leq 0 \\
s/ \lambda\; \textrm{ if } s> 0.
\end{cases}
\end{align}

The equations $\M^+ (D^2 u)+r^a u^p=0$ and $\M^- (D^2 u)+r^a u^p=0$, for $r\neq  0$, in radial coordinates for positive solutions then become, respectively,
\begin{align}\label{P radial m}\tag{$P_+$}
\textrm{$u^{\prime\prime}\;=\; M_+(-r^{-1}(N-1)\, m_+(u^\prime)- r^a u^p )$,\quad $u>0$};
\end{align}
\vspace{-0.8cm} 
\begin{align}\label{P radial m-}\tag{$P_-$}
\textrm{$u^{\prime\prime}\;=\; M_-(-r^{-1}(N-1)\, m_-(u^\prime)- r^a u^p )$,\quad $u>0$},
\end{align}
which are understood in the maximal interval where $u$ is positive.

\subsection{The new variables and the quadratic system}\label{section 2.1}

Let $u$ be a positive solution of \eqref{P radial m} or \eqref{P radial m-}. Thus we can define the new functions
\begin{align}\label{X,Z a}
X(t)=-\frac{ru^\prime(r)}{u(r)}\,,  \;\;\quad Z(t)=-\frac{r^{1+a}\, u^p(r)}{u^\prime(r)} \, \quad \textrm{ for }\,t=\mathrm{ln}(r),
\end{align}
whenever $r>0$ is such that $u(r)\neq  0$ and $u^\prime (r)\neq 0$.

We consider the phase plane $(X,Z)\in \real^2$. 
Since we are studying positive solutions, the points $(X(t),Z(t))$ belong to the first quadrant when $u^\prime<0$; or to the third quadrant when $u^\prime>0$.
We denote the first and third quadrants by $1Q$, $3Q$ respectively.

As a consequence of this monotonicity, the problems \eqref{P radial m} and \eqref{P radial m-} become in $1Q$:
\begin{align}\label{P M+ 1Q}
\textrm{for $\M^+$ :\qquad$u^{\prime\prime}\;=\; M_+(-\lambda r^{-1}(N-1)u^\prime- r^a u^p )$,\quad $u>0$\quad in $1Q$},
\end{align}
\vspace{-1cm}
\begin{align}\label{P M- 1Q}
\textrm{\;\;for $\M^-$ :\qquad$u^{\prime\prime}\;=\; M_-(-\Lambda r^{-1}(N-1)u^\prime- r^a u^p )$,\quad $u>0$\quad in $1Q$}.
\end{align}
On the other hand, since $u^\prime>0$ implies $u^{\prime\prime}<0$, one finds out in $3Q$:
\begin{align}\label{P M+ 3Q}
\textrm{for $\M^+$ :\qquad$\lambda u^{\prime\prime}\;=\; -\Lambda r^{-1}(N -1)u^\prime-  r^a u^p $,\quad $u>0$\quad in $3Q$},
\end{align}
\vspace{-1cm}
\begin{align}\label{P M- 3Q}
\textrm{for $\M^-$ :\qquad$\Lambda u^{\prime\prime}\;=\; -\lambda r^{-1}(N -1)u^\prime-  r^a u^p $,\quad $u>0$\quad in $3Q$}.
\end{align}

In terms of the functions \eqref{X,Z a}, we derive the following autonomous dynamical systems:
\begin{align}\label{DS+}
\textrm{in $1Q$, \quad}
\left\{
\begin{array}{ccl}
\dot{X} &=& \;X \; [\,X+1-M_+ (\lambda(N-1)-Z )\,], \\
\dot{Z} &=& \;Z\; [\,1+a-pX +M_+ (\lambda(N-1)-Z )\,];
\end{array}
\right.
\end{align}
\vspace{-0.7cm}
\begin{align}\label{DS+ 3Q}
\textrm{in $3Q$, \quad}
\dot{X} = \;X \; [\,X-(\tilde{N}_- -2)+{Z}/{\Lambda} \,] , \;\;
\dot{Z}= \;Z\; [\, \tilde{N}_- +a-pX - {Z}/{\Lambda}\,],
\end{align}
corresponding to \eqref{P M+ 1Q}, \eqref{P M+ 3Q} for $\M^+$, where the dot $\dot{ }$ stands for $\frac{\mathrm{d}}{\mathrm{d} t}$.
Similarly one has
\begin{align}\label{DS-}
\textrm{in $1Q$, \quad}
\left\{
\begin{array}{ccl}
\dot{X} &=& \;X \; [\,X+1-M_- (\Lambda(N-1)-Z )\,], \\
\dot{Z} &=& \;Z\; [\,1+a-pX +M_- (\Lambda(N-1)-Z )\,];
\end{array}
\right.
\end{align}
\vspace{-0.7cm}
\begin{align}\label{DS- 3Q}
\textrm{in $3Q$, \quad}
\dot{X} = \;X \; [\,X-(\tilde{N}_+ -2)+{Z}/{\lambda} \,] , \;\;
\dot{Z}= \;Z\; [\, \tilde{N}_+ +a-pX - {Z}/{\lambda}\,],
\end{align}
associated to \eqref{P M- 1Q}, \eqref{P M- 3Q} for $\M^-$.

We stress that \eqref{DS+}, \eqref{DS-} correspond to positive, decreasing solutions of \eqref{P radial m}, \eqref{P radial m-}. We will see in Section \ref{section classification} that this holds for regular and singular solutions of \eqref{P} in $\rN$ or in a ball.

\smallskip

On the other hand, given any trajectory $\tau=(X,Z)$ of \eqref{DS+}-\eqref{DS- 3Q} either in $1Q$ or $3Q$, we define
\begin{align}\label{def u via X,Z}
\textstyle{ u(r)=r^{-\alpha} (XZ)^{\frac{1}{p-1}}(t) }, \quad\textrm{ where }\, r=e^t.
\end{align}
Then we deduce
\begin{align*}
\textstyle{ u^\prime (r)= -\alpha r^{-\alpha-1} (XZ)^{\frac{1}{p-1}}(t) + \frac{r^{-\alpha}}{p-1} (XZ)^{\frac{1}{p-1}-1}(t) \,\frac{\dot{X} Z+ X\dot{Z} }{r} 	=  -X r^{-\alpha-1} (XZ)^{\frac{1}{p-1}}(t)=-\frac{X(t)u(r)}{r}},
\end{align*}
from which we recover \eqref{X,Z a}. Since $X\in C^1$, then $u\in C^2$. From this, one immediately sees that $u$ satisfies either \eqref{P radial m} or \eqref{P radial m-} from the respective equations for $\dot{X},\dot{Z}$ in the dynamical system.	

\smallskip

An important role in the study of our problem is played by the lines $\ell_\pm$, defined by
\begin{align}\label{ell}
\ell_+=\{(X,Z):\;Z=\lambda(N-1)\}\cap 1Q \;\;\;\textrm{ for \;}\M^+,
\\
\ell_-=\{(X,Z):\;Z=\Lambda(N-1)\}\cap 1Q \;\;\;\textrm{ for\; }\M^-.\nonumber
\end{align}
For each of the two systems \eqref{DS+} and \eqref{DS-} respectively, the lines $\ell_\pm$ splits $1Q$ into two regions, up and down: 
\begin{align}\label{R+}
R^+_\lambda=\{(X,Z):Z>\lambda (N-1)\}\cap 1Q, \quad
R^-_\lambda=\{(X,Z):Z<\lambda (N-1)\}\cap 1Q  \;\;\;\textrm{ for }\M^+,
\end{align}
\vspace{-0.6cm}
\begin{align}\label{R-}
R^+_\Lambda=\{(X,Z):Z>\Lambda (N-1)\}\cap 1Q, \quad
R^-_\Lambda=\{(X,Z):Z<\Lambda (N-1)\}\cap 1Q  \;\;\;\textrm{ for }\M^-.
\end{align}

\smallskip

In terms of $(P_\pm)$, $\ell_\pm$ is the line where a decreasing solution $u$ changes concavity in the sense that, when $(X(t),Z(t))\in R^+_\lambda$ (or $R^+_\Lambda$) then the corresponding solution $u$ through the transformation \eqref{def u via X,Z} is concave, while for $(X(t),Z(t))\in R^-_\lambda$ (or $R^-_\Lambda$), $u$ is convex. Hence, these regions are essential to determine the precise expressions of \eqref{P radial m} and \eqref{P radial m-} according to $M_+$ and $M_-$ in \eqref{m,M+}, \eqref{m,M-}.
For instance, when $u^{\prime\prime} \leq 0$ and $u^\prime<0$, \eqref{P radial m} becomes
\begin{center}
$-u^{\prime\prime} (r)-\frac{N-1}{r}u^\prime (r)=r^a \frac{u^p (r)}{\lambda}$,
\end{center}
for which the left hand side is the standard radial Laplacian operator; while for $u^{\prime\prime}> 0$ and $u^\prime<0$, \eqref{P radial m} reads as
\begin{center}
	$-u^{\prime\prime} (r)-\frac{\tilde{N}_+-1}{r}u^\prime (r)=r^a \frac{u^p (r)}{\Lambda}$,
\end{center}
where, in turn, the LHS is the Laplacian in the noninteger dimension $\tilde{N}_+$, see \eqref{def dimensional-like}.
Analogously one treats $\M^-$.
Note that in $3Q$ we always obtain Laplacian operators in dimensions $\tilde{N}_-$ (for $\M^+$) and $\tilde{N}_+$ (for $\M^-$), see \eqref{P M+ 3Q}, \eqref{P M- 3Q}. Our system is then the union of equations driven by different Laplacian-like operators. 
This explains the difficulty in dealing with fully nonlinear operators.

We stress that Lane-Emden-Henon problems for Laplacian operators were already studied in \cite{BV} in terms of the dynamical system \eqref{DS+} in the case  $\lambda=\Lambda=1$ subject to the transformation \eqref{X,Z a}.

\smallskip

At this stage it is worth observing that the systems \eqref{DS+} and \eqref{DS-} are continuous on $\ell_\pm$. More than that, the right hand sides are locally Lipschitz functions of $X,Z$, so the usual ODE theory applies. 
That is, one recovers existence, uniqueness, and continuity with respect to initial data as well as continuity with respect to the parameter $p$.

\subsection{Stationary points and local analysis}\label{section local}

We start this section investigating the sets where $\dot{X}=0$ and $\dot{Z}=0$. Let us focus our analysis on $1Q$, since the only stationary point on the boundary of $3Q$ is the origin.
One writes the dynamical systems \eqref{DS+} and \eqref{DS-} in terms of the following ODE first order autonomous equation
\begin{align}\label{Prob x=(X,Z)}
\dot{x}=F(x), \quad\textrm{where }\quad x=(X,Z), \quad F(x):=(f(x),g(x)). 
\end{align}
with $\dot{x}=(\dot{X},\dot{Z})$. For instance, in the case of the operator $\M^+$, then $f,g$ are given by
\begin{center}
$f(x)=
\begin{cases}
 X(X - ({N}-2) + \frac{Z}{\lambda}) \;\textrm{ in } R^+_\lambda\\
  X(X - (\tilde{N}_+-2) + \frac{Z}{\Lambda})\;\textrm{ in } R^-_\lambda
\end{cases}
,\;\;\;
g(x)=
\begin{cases}
 Z({N}+a - pX - \frac{Z}{\lambda}) \;\textrm{ in } R^+_\lambda\\
 Z(\tilde{N}_++a - pX - \frac{Z}{\Lambda})\;\textrm{ in } R^-_\lambda.
\end{cases}$
\end{center}

We first recall some standard definitions from the theory of dynamical systems.
\begin{defin}
A stationary point $Q$ of \eqref{Prob x=(X,Z)} is a zero of the vector field $F$.
If $\sigma_1$ and $\sigma_2$ are the eigenvalues of the Jacobian matrix $DF(Q)$, then $Q$ is hyperbolic if both $\sigma_1, \sigma_2$ have nonzero real parts. If this is the case, $Q$ is a \textit{source} if $\mathrm{Re}(\sigma_1),\mathrm{Re}(\sigma_2)>0$, and a \textit{sink} if $\mathrm{Re}(\sigma_1),\mathrm{Re}(\sigma_2)<0$; $Q$ is a \textit{saddle point} if $\mathrm{Re}(\sigma_1)<0<\mathrm{Re}(\sigma_2)$.
\end{defin}

Next we recall an important result from the theory of dynamical systems which describes the local stable and unstable manifolds near saddle points of the system \eqref{Prob x=(X,Z)}; see \cite[theorems 9.29, 9.35]{Hale}.
Here the usual theory for autonomous planar systems applies since each stationary point $Q$ possesses a neighborhood which is strictly contained in $R^+_\lambda$ or $R^-_\lambda$ where the vector field $F$ is $C^1$. 

\begin{prop}\label{prop GH}
Let $Q$ be a saddle point of \eqref{Prob x=(X,Z)}. 
Then the local stable (resp.\ unstable) manifold at $Q$ is locally a $C^1$ graph over the stable (resp.\ unstable) line of the linearized vector field.
In this case, if the linearized system has a stable line direction $L$, then there exists exactly two trajectories $\tau_1$ and $\tau_2$ arriving at $Q$ which admit the same tangent at the point $Q=\overline{\tau}_1\cap \overline{\tau}_2$ given by $L$.
Analogously there are only two trajectories coming out from $Q$ with the same property.
\end{prop}

We sometimes use the following notation to describe the limit of trajectories in the phase plane.
\begin{defin}[$\alpha$ and $\omega$ limits]\label{alpha and omega limits}
We call $\alpha$-limit of the orbit $\tau$, and we denote it by $\alpha(\tau)$, as the set of limit points of $\tau(t)$ as $t\to -\infty$.
Similarly one defines $\omega(\tau)$ i.e.\ the $\omega$-limit of $\tau$ at $+\infty$.
\end{defin}

We observe that both $X$ and $Z$ axes are invariant by the flow.
In particular, each quadrant is an invariant set for the dynamics.
Moreover, let us keep in mind the following segments in the plane $(X,Z)$. For the system \eqref{DS+}, we define
\begin{align}\label{ell 1+}
\ell_1^+ = \{\,(X,Z):\, Z = \Lambda (\tilde{N}_+-2)-\Lambda X\,\}\cap 1Q
\end{align}
which is the set where $\dot{X}=0$ and $X>0$; also
\begin{align}\label{ell 2+}
\ell_2^+ &=
{\ell_{2+}^+}\; \cup\;{\ell_{2-}^+}
\end{align}
with ${\ell_{2+}^+}= \{(X,Z): Z = \lambda ({N}+a-p X)\}\cap R^+_\lambda$,
${\ell_{2-}^+}= \{(X,Z): Z = \Lambda (\tilde{N}_++a- p X)\}\cap R^-_\lambda$,
which is the set where $\dot{Z}=0$ and $Z>0$; see Figures \ref{Fig flow} and \ref{Fig flow2}.

\smallskip

Notice that $\ell_1^+$ is a segment entirely contained in $R^-_\lambda$, since there are no other points in $1Q$ where $\dot{X}=0$ in the interior of the region $R^+_\lambda$.
Moreover, \eqref{ell 2+} is the union of two segments which join at the point $(\frac{1+a}{p},\lambda (N-1))\in \ell_+\cap\overline{\ell}_2^+$, see Figure \ref{Fig flow}.
The analogous sets for $\M^-$ are
\begin{align}\label{ell 1-}
\ell_1^- = \{\,(X,Z):\, Z = \lambda (\tilde{N}_--2)-\lambda X\,\}\cap 1Q 
\end{align}
which is the set where $\dot{X}=0$ and $X>0$ (contained in $R^-_\Lambda$); and
\begin{align}\label{ell 2-}
\ell_2^- =
{\ell_{2+}^-}\; \cup\;{\ell_{2-}^-}
\end{align}
with
${\ell_{2+}^-}= \{(X,Z): Z = \Lambda ({N}_-+a- p X)\}\cap R^+_\Lambda$\,, 
${\ell_{2-}^-}= \{(X,Z): Z = \lambda (\tilde{N}_-+a- p X)\}\cap R^-_\Lambda$,
which is the set where $\dot{Z}=0$ and $Z>0$.

\begin{lem}\label{stationary M+}
The stationary points of the dynamical systems \eqref{DS+}--\eqref{DS- 3Q} are:
	\begin{center}
for $\M^+$:\quad	$O=(0,0)$, \quad $N_0=(0,\lambda N+\lambda a)$, \quad $A_0=(\tilde{N}_+-2,0)$, \quad $M_0=(X_0,Z_0)$,
	\end{center}
where $X_0 = \alpha$, and $Z_0 =\Lambda(\tilde{N}_+-p\alpha+a) = \Lambda(\tilde{N}_+-2 - \alpha)$, see Figure \ref{Fig flow};
\begin{center}
for $\M^-$:\quad	$O=(0,0)$, \quad $N_0=(0,\Lambda N+\Lambda a)$, \quad $A_0=(\tilde{N}_--2,0)$, \quad $M_0=(X_0,Z_0)$,
\end{center}
where $X_0 = \alpha$ and $Z_0 = \lambda(\tilde{N}_--p\alpha+a) = \lambda(\tilde{N}_--2 - \alpha)$.
\end{lem}

\begin{proof} We just show the $\M^+$ case.
First notice that the system does not admit stationary points in $3Q$ nor on the line $\ell_+$. In the region $\overline{R}^+_\lambda$ we have already seen that $\dot{X}=0$ implies $X=0$, since $\ell_1^+$ does not intersect $R^+_\lambda$. By $\dot{Z}=0$ we obtain $Z=\lambda(N+a-pX)$ since $Z\neq 0$ in $R^+_\lambda$. Hence we reach the equilibrium point $N_0$.
In $\overline{R}^-_\lambda$, from $\dot{X}=0$ we have either $X=0$ or $Z=\Lambda(\tilde{N}_+-2-X)$, while by $\dot{Z}=0$ we deduce that either $Z=0$ or $Z=\Lambda(\tilde{N}_++a-pX)$.
Therefore we obtain the points $O$, $A_0$, $M_0$, and $(0,\Lambda (\tilde{N}_++a))$.
However, the latter does not belong to $\overline{R}^-_\lambda$ as long as $a>-1$. 
\end{proof}

Next we analyze the directions of the vector field $F$ in \eqref{Prob x=(X,Z)} on the $X,Z$ axes, on the concavity line
$\ell_\pm$, and also on $\ell_1^\pm$,  $\ell_2^\pm$; see \eqref{ell}, and \eqref{ell 1+}--\eqref{ell 2-}.

\begin{prop}\label{prop flow} The systems \eqref{DS+} and \eqref{DS-} enjoy the following properties (see Figures \ref{Fig flow},\ref{Fig flow2}):

\begin{enumerate}[(1)]
	\item Every trajectory of \eqref{DS+} in $1Q$ crosses the line $\ell_+$ transversely except at the point
	$P=(\frac{1+a}{p},\lambda (N-1))$. Moreover, it passes from 
	$R^+_\lambda$ to $R^-_\lambda$  if $X> \frac{1+a}{p}$, while it moves from $R^-_\lambda$ to $R^+_\lambda$ if $X< \frac{1+a}{p}$. The vector field at $P$ always points to the right.
	A similar statement holds for $\M^-$ via the system \eqref{DS-} considering respectively $\ell_-$, $(\frac{1+a}{p},\Lambda (N-1))$, $R^+_\Lambda$, $R^-_\Lambda$;
	
	\item The flow induced by \eqref{Prob x=(X,Z)} on the $X$ axis points to the left for $X\in (0, \tilde{N}_\pm-2)$, and to the right when $X> \tilde{N}_\pm-2$. On the $Z$ axis it moves up between $O$ and $N_0$, and down above $N_0$;
	
	\item The vector field $F$ on the line $\ell_1^\pm$ is parallel to the $Z$ axis whenever $X\neq \alpha$. It points up if $X< \alpha$, and down if $X>\alpha$. 
	Further, on the set $\ell_2^\pm$ the vector field $F$ is parallel to the $Z$ axis for $X\neq \alpha$. It moves to the right if $X< \alpha$, and to the left if $X>\alpha$.
\end{enumerate}

\end{prop}

\begin{proof}
\textit{(1)} We just observe that $\dot{X}=X(X+1)>0$, and $\dot{Z}=Z\,(1+a-pX)$ on $\ell_\pm$.

\textit{(2)} For instance consider $\M^+$. Since the $X$ axis is contained in $R^-_\lambda$, then $\dot{X}=X(X-(\tilde{N}_+-2))$ which is positive for $X<\tilde{N}_+-2$ and negative for $X>\tilde{N}_+-2.$
Now, $\dot{Z}= Z(N+a-{Z}/{\lambda})$ in $R^+_\lambda$ is positive if $Z<\lambda (N+a)$ and negative for $Z>\lambda (N+a)$.	On the other hand, $\dot{Z}= Z(\tilde{N}_++a-{Z}/{\Lambda})>0$ in $R^-_\lambda$, since $\Lambda (\tilde{N}_++a)>\lambda (N-1)$ for $a>-1$.

\textit{(3)} Notice that 
$\dot{Z}=(p-1)Z(\alpha-X)$ on $\ell_1^\pm$ and  $\dot{X} =(p-1)X(\alpha-X)$ on $\ell_2^\pm$.
Both are positive quantities for $X<\alpha$, and negative when $X> \alpha$.
\end{proof}

\begin{rmk}\label{remark concavity}
It is not difficult to see that an orbit can only reach the point $P$ in Proposition \ref{prop flow} (1) from $R^-_\lambda$. In fact by $\dot{X}>0$ and the inverse function theorem, $Z$ is a function of $X$ on $\ell_+$. Next, an analysis of the continuous function $X\mapsto\partial^2_X Z(X)$ at $P$ shows that an orbit passing through $P$ has a local maximum at this point; see the vector field at that point in Figures \ref{Fig flow} and \ref{Fig flow2}.
\end{rmk}

\begin{figure}[!htb]
	\centering
	\includegraphics[scale=0.6]{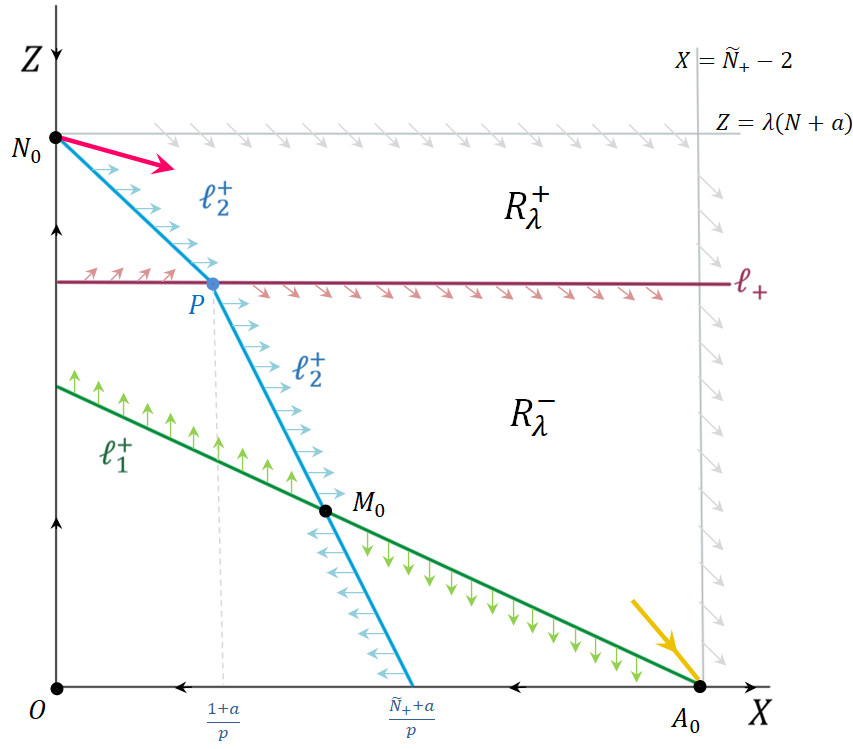}
	\caption{The flow behavior in $1Q$ for $\M^+$ when $p>p^{s,a}_+$.}
	\label{Fig flow}
\end{figure}

The next proposition gathers the crucial dynamics at each stationary point. 

\begin{prop}[$\M^\pm$]\label{local study stationary points}
The following properties are verified for the systems \eqref{DS+} and \eqref{DS-},
	\begin{enumerate}
\item The origin $O$ is a saddle point. The stable and unstable directions of the linearized system are the $X$ and $Z$ axes respectively;
		
\item $N_0  $ is a saddle point. The tangent unstable direction is parallel to the line 
\begin{center}
	$Z= \frac{-p\lambda (N+a)}{N+2+2a} X$\; if the operator is $\M^+$,\quad 
	$Z= \frac{-p\Lambda (N+a)}{N+2+2a} X$\; for $\M^-$;
\end{center}

\item $A_0 $ is a saddle point for $p>p_\pm^{s,a}$. The linear stable direction is parallel to the line 
\begin{center}
$Z= \frac{-p(\tilde{N}_+ - 2) +2+a}{\tilde{N}_+ - 2} \Lambda X$\; in the case of $\M^+$,
\quad	$Z= \frac{-p(\tilde{N}_- - 2) +2+a}{\tilde{N}_- - 2} \lambda X$\; for $\M^-$,
\end{center}
while the unstable tangent direction lies on the $X$ axis.	For $p< p_\pm^{s,a}$ $A_0$ is a source;

At $p=p^{s,a}_\pm$  $A_0$ coincides with $M_0$ and belongs to the $X$ axis. In this case, it is not a hyperbolic stationary point.

\item For $p<p^{s,a}_\pm$ $M_0$ belongs to the fourth quadrant. Also, $M_0\in 1Q \,\Leftrightarrow\, p>p^{s,a}_\pm$ in which case:
\vspace{-0.1cm}
\begin{enumerate}[(i)]
	\item $M_0 $ is a source if $p_\pm^{s,a}<p<p_\pm^{p,a}$;
	
	\item $M_0$ is a sink for $p>p_\pm^{p,a}$;
	
	\item $M_0$ is a center at $p=p^{p,a}_\pm$.
\end{enumerate}
	\end{enumerate}
\end{prop}

\begin{figure}[!htb]
	\centering
	\includegraphics[scale=0.5]{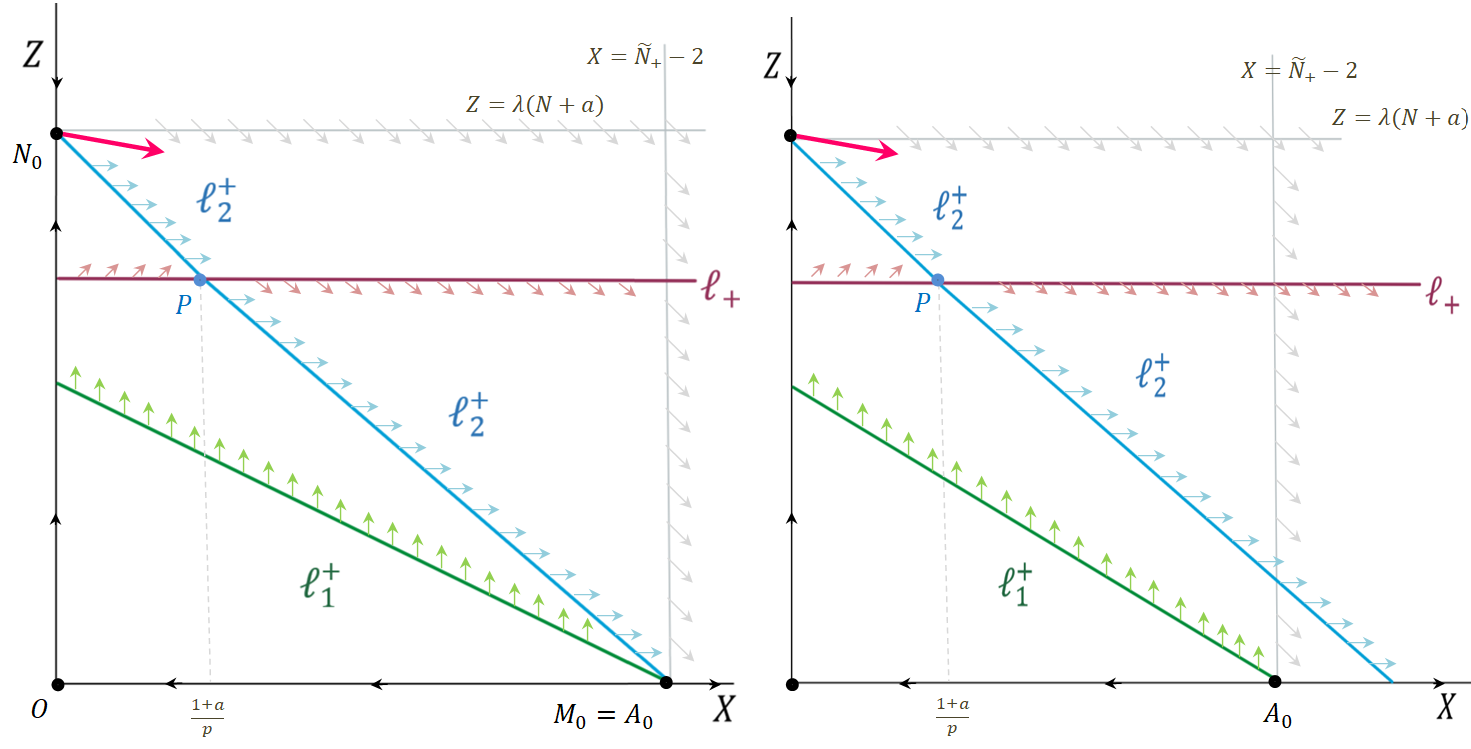}
	\caption{The flow behavior in $1Q$ for $\M^+$ when $p= p^{s,a}_+$ (LHS), and $p< p^{s,a}_+$ (RHS).}
	\label{Fig flow2}
\end{figure}

The dynamics at each stationary point depends upon the linearization of the system \eqref{DS+}. Since the point $N_0$ belongs to $R^+_\lambda$ where the system corresponds to the Henon equation for the standard Laplacian, we could just refer to \cite{BV} for the local analysis of $N_0$, as long as $p> p^{p,a}_\pm$. The other points $O,N_0,A_0$ instead belong to $R^-_\lambda$ where the system now corresponds to the Henon equation for the Laplacian in dimension $\tilde{N}_\pm$. In this last case some variations with respect to \cite{BV} are needed.
Since the classification of stationary points is the heart of our analysis and for reader's convenience, details are provided in the Appendix.

On the other hand, in both cases, some deeper analysis is required when $p\leq p^{p,a}_\pm$. We treat this case in Proposition \ref{p=p^s} by using the dynamics of the system.  
For $p=p^{p,a}_\pm$ we refer to the Appendix, see Proposition \ref{prop M0 center}. 

\medskip

To finish the section, a local uniqueness result follows directly from Propositions \ref{prop GH} and \ref{local study stationary points}.
\begin{prop}\label{local uniqueness}
For every $p>1$ there is a unique trajectory coming out from $N_0$ at $-\infty$, which we denote by $\Gamma_p$. Further, for $p>p^{s,a}_\pm$ there exists a unique trajectory arriving at $A_0$ at $+\infty$ that we denote by $\Upsilon_{p}$. In terms of Definition \ref{alpha and omega limits},
	\begin{align}\label{Gammap}
	\textrm{for all $p>1$,\qquad $\Gamma_p$\;\; is such that\quad $\alpha(\Gamma_p)=N_0$;}
	\end{align}
	\vspace{-1cm}
	\begin{align}\label{Upsilonp}
	\textrm{for all $p>p^{s,a}_\pm$, \qquad  $\Upsilon_p$ \;\;is such that\quad $\omega(\Upsilon_p)=A_0$.}
	\end{align} 
\end{prop}

\begin{rmk}\label{graph}
Notice that these trajectories uniquely determine the global unstable and stable manifolds of the stationary points $N_0$ and $A_0$ respectively.  In particular, by Proposition \ref{prop GH} they are graphs of functions in a neighborhood of the stationary points in their respective ranges of $p$.
The tangent directions at $N_0,A_0$ are displayed together in Figure \ref{Fig flow} for $p> p^{s,a}_\pm$. In fact, they both belong to the region where $\dot{X}>0$ and $\dot{Z}<0$ in their respective ranges of $p$.
\end{rmk}

\subsection{Periodic orbits}\label{section periodic}

In this section we continue investigating the limit sets of the trajectories. Let us see in which intervals of $p$ the dynamical systems \eqref{DS+} and \eqref{DS-} allow the existence of periodic orbits.

The Poincaré-Bendixson theorem \cite{Hale} for planar autonomous systems says that the only admissible $\omega$ and $\alpha$ limits of bounded trajectories are either a stationary point or a periodic orbit. 

Observe that we have the following ordering for the exponents $p^{p,a}_-<p_\Delta^a < p^{p,a}_+$ defined in \eqref{critical exponents a}.

\begin{prop}[Dulac's criterion]\label{Dulac}
Let $\lambda<\Lambda$. 
For $\M^+$ there are no periodic orbits of \eqref{DS+} when $1<p\leq  p_\Delta^a$ or $p> p^{p,a}_+$. For $\M^-$ no periodic orbits of \eqref{DS-} exist if $1<p<p^{p,a}_-$ or $p\geq p_\Delta^a$. Moreover, 
\vspace{-0.2cm}
	\begin{enumerate}[(i)]
		\item there are no periodic orbits strictly contained in $R^+_\lambda\cup \ell_+$ $($resp.\ $(R^+_\Lambda\cup \ell_-)$, for any $p>1$;
\vspace{-0.2cm}
		\item periodic orbits contained in $R^-_\lambda\cup \ell_+$ $($resp.\ $R^-_\Lambda\cup \ell_-)$ are admissible only at $p=p^{p,a}_\pm$. 
Also, no periodic orbits at $p^{p,a}_\pm$ can cross the concavity line $\ell_\pm$ twice.
	\end{enumerate}
\end{prop}

\begin{proof}
Define $\varphi(X,Z) = X^\alpha Z^\beta$, where 
	$\beta=\frac{3-p}{p-1}$ and $\alpha$ as in \eqref{critical exponents a}. Set
	$	\Phi(X,Z) = \partial_X(\varphi f) + \partial_Z(\varphi g)	$, with $f$ and $g$ defined in \eqref{Prob x=(X,Z)}.
	For $\M^+$ we have
\vspace{-0.2cm}
	\begin{align*}
	\Phi(X,Z) 
	&= 
	\begin{cases}
\,X^\alpha Z^\beta \,[\,\alpha (X-(N-2) + \frac{Z}{\lambda}) + \beta (N+a - pX - \frac{Z}{\lambda}) + (2-p)X + 2 - \frac{Z}{\lambda}\,]\;\; \textrm{ in\;$R^+_\lambda$,}
	\\
\,X^\alpha Z^\beta \,[\,\alpha (X-(\tilde{N}_+-2) + \frac{Z}{\Lambda}) + \beta (\tilde{N}_++a - pX - \frac{Z}{\Lambda}) + (2-p)X + 2 - \frac{Z}{\Lambda}\,]\;\; \textrm{ in\;$R^-_\lambda$,}
	\end{cases}
	\\
	&=  
	\begin{cases}
\,{\varphi (X,Z)}{(p-1)^{-1}}\,[ \,-p(N-2) + (N+2+2a)\,] \quad \textrm{ in\; $R^+_\lambda$,}
	\\
\,{\varphi (X,Z)}{(p-1)^{-1}}\,[ \,-p(\tilde{N}_+-2) + (\tilde{N}_++2+2a)\, ] \quad \textrm{ in\; $R^-_\lambda$.}
	\end{cases}
	\end{align*}
Both expressions are positive if $1<p<\min(p^{p,a}_+,p_\Delta^a) = p_\Delta^a$; and both are negative if $p>\max(p^{p,a}_+,p_\Delta^a) = p^{p,a}_+$.
Then one concludes by the same argument as in the classical Bendixson–Dulac criterion, see also \cite[Theorem 3.1]{GQsingular}.
Indeed, the vector field $F=(f,g)$ is Lipschitz continuous in $(X,Z)$, so Green's area formula for the domain $D$ enclosed by a periodic trajectory applies such as
\begin{align}\label{eq Dulac}
\textstyle{\int_{\partial D} \varphi \,\{ f \,\rmd Z - g\,\rmd X \}  = 
		\int_D \Phi (X,Z) \,\rmd X \rmd Z=
		\int_{R^+_\lambda\cap D}\Phi (X,Z)\, \rmd X \rmd Z + \int_{R^-_\lambda\cap D}\Phi (X,Z)\, \rmd X \rmd Z}.
\end{align}
The RHS is nonzero for $p\in (1,p_\Delta^a)\cup (p^{p,a}_+,\infty)$, but the LHS is zero because $\rmd X=f \rmd t$, $\rmd Z=g\rmd t$. 
Further, at $p^a_\Delta$ one has $\Phi=0$ in $R^+_\lambda$ and so the first integral in \eqref{eq Dulac} (in the RHS) vanishes, while the second one is positive.
For $\M^-$ the computations are similar by using that $\min(p^{p,a}_-,p_\Delta^a) = p^{p,a}_-$ and  $\max(p^{p,a}_-,p_\Delta^a) = p_\Delta^a$.

Next we look at the interval $[p^a_\Delta, p^{p,a}_+]$ for $\M^+$.
Note that Poincaré-Bendixson theorem guarantees the existence of a stationary point in the domain $D$ inside a periodic orbit.
Since the only admissible stationary point in the interior of $1Q$ is $M_0\in R^-_\lambda$ for $p>p^{s,a}_+$, while for $p\leq p^{s,a}_+$ $M_0$ is not an option (see Figure \ref{Fig flow2}), then $(i)$ follows. 

To prove $(ii)$ let us observe that if a periodic orbit is contained in $R^-_\lambda\cup \ell_+$ then by Proposition~\ref{prop flow} (1) it may intersect the line $\ell_+$ only at one point, namely the point $P$. Hence we can repeat the previous argument, neglecting the integral expression in $R^+_\lambda$. Then we get that there are no periodic orbits in $R^-_\lambda\cup \ell_+$ for every $p\neq p^{p,a}_+$.
To finish, if a periodic orbit existed which crossed twice the line $\ell_+$ at $p^{p,a}_+$, then the first integral of \eqref{eq Dulac} (in the RHS) would be positive, while the second one is equal to zero because $\Phi=0$ in $R_\lambda^-$.
The case for $\M^-$ and $p^{p,a}_-$ is analogous.
\end{proof}

Notably Dulac's criterion brings out the critical exponents $p_\Delta^a$ and $p^{p,a}_\pm$.
They correspond to critical exponents for the two Laplacian operators $\Delta_N$ and $\Delta_{{\tilde{N}}_\pm}$, in dimensions $N$ and $\tilde{N}_\pm$.

Other limit cycles $\theta$ are admissible by the dynamical system as far as they cross $\ell_\pm$ twice. They do appear for $\M\pm$ as we shall see in Sections \ref{section Pucci+}, \ref{section Pucci-}. This happens because Dulac's criterion is inconclusive in a whole interval when $\lambda<\Lambda$. Formally, the Pucci problem opens space for new periodic orbits in order to appropriately glue both Laplacian operators.

\subsection{A priori bounds}\label{section AP bounds}

We prove ahead important bounds for trajectories of \eqref{DS+} or \eqref{DS-} which are defined for all $t$ in intervals of type $(\hat{t},+\infty)$ or $(-\infty,\hat{t})$. 

By Poincaré-Bendixson theorem, if a trajectory of \eqref{DS+} or \eqref{DS-} does not converge to a stationary point neither to a periodic orbit, either forward or backward in time, then it necessarily blows up in that direction. 
In the next propositions we prove that a blow up may only occur in finite time. The first result is obtained in the first quadrant.

\begin{prop}\label{AP bounds}
	Let $\tau$ be a trajectory of \eqref{DS+} or \eqref{DS-} in $1Q$, with $\tau(t)=(X(t),Z(t))$ defined for all $t\in(\hat{t},+\infty)$, for some $\hat{t}\in \real$. Then
\vspace{-0.2cm}
	\begin{align}\label{X< tildeN-2}
\textrm{$X(t)< \tilde{N}_\pm-2$,  \quad for all\; $t\geq \hat{t}$.}
	\end{align}
If instead, $\tau$ is defined for all $t\in(-\infty,\hat{t})$, for some $\hat{t}\in \real$, then
\begin{align}\label{Z< lambdaN}
\textrm{$Z(t)<\lambda (N+a)$ in the case of $\M^+$,\quad $Z(t)<\Lambda (N+a)$ for $\M^-$,\quad for all\; $t\leq \hat{t}$.}
\end{align}
In particular, if a global trajectory is defined for all $t\in \real$ in $1Q$ then it remains inside the box $(0,\tilde{N}_+-2)\times (0,\lambda (N+ a))$ in the case of $\M^+$, while it stays in $(0,\tilde{N}_--2)\times (0,\Lambda (N+ a))$ for $\M^-$.
\end{prop}

\begin{proof}
Let us first prove \eqref{X< tildeN-2} when the operator is $\M^+$, $\tilde{N}_+\leq N$.
Arguing by contradiction we assume that for some $t_1\geq \hat{t}$ we have $X(t_1)\geq \tilde{N}_+-2$.
Notice that $\dot{X}>0$ on the half line $L^+=\{(X,Z):X=\tilde{N}_+-2\}\cap 1Q$, see \eqref{ell 1+}.
Therefore $X(t)>X(t_1)\geq \tilde{N}_+-2$ for all $t> t_1$.

\smallskip

We claim that $X(t)\rightarrow+\infty$ as $t\to +\infty$. To see this, first notice that $Z$ is bounded from $t_1$ on, since $\dot{Z}<0$ to the right of $L^+$, see \eqref{ell 2+}.
If we had $X(t) \leq C$ for some $C>0$ for $t\geq t_1$, then $\tau$ would be a bounded trajectory from $t_1$ on.
Then by Poincaré-Bendixson it should converge to a stationary point as $t\to +\infty$. Notice that periodic orbits are not allowed to the right of $L^+$ by the direction of the vector field, see Figure \ref{Fig flow}. This proves the claim, since no stationary points exist on the right of $L^+$.

Thus, we can pick a time $t_2$ such that $X(t_2)>N-2\geq \tilde{N}_+-2$. Again by monotonicity, $X(t)>\tilde{N}_+-2$ for all $t\geq t_2$.

Then we have two cases: either the trajectory $\tau$ reaches the region $R^-_\lambda$ for some $t_3\geq t_2$, or it stays in $R^+_\lambda$ for all time.
If the first holds, then $\tau(t)$ remains there for all $t\geq t_3$, since $\dot{Z}<0$ to the right of $L^+$, see Figure \ref{Fig flow}.
Observe that the first equation in \eqref{DS+} yields 
$\frac{\dot{X}}{X[X-(\tilde{N}_+-2)]}\geq 1$. Moreover,
\begin{align}\label{int1F X}
\textstyle{\frac{(\tilde{N}_+-2) \dot{X} }{X[X-(\tilde{N}_+-2)]}
	=\frac{\dot{X}}{X-(\tilde{N}_+-2)}-\frac{\dot{X}}{X}
	=\frac{\rmd}{\rmd t}\,\mathrm{ln} \left(  \frac{X(t)-\tilde{N}_++2}{X(t)}  \right) \;\; \textrm{ for all } t\geq t_3.}
\end{align}
Therefore, by integrating \eqref{int1F X} in the interval $[t_3,t]$ we get
\begin{align}\label{int2F X}
\textstyle{X(t)\geq \frac{\tilde{N}_+-2}{1-c e^{(\tilde{N}_+-2)(t-t_3)} }}, \;\; \textrm{ where }\;\;\textstyle{c=1-\frac{\tilde{N}_+-2}{X(t_3)}\in (0,1)},
\end{align}
and in particular $X$ blows up in the finite time $t_1=t_3+\frac{\mathrm{ln}(1/c)}{\tilde{N}_+-2}$.

If instead $\tau$ stays in $R^+_\lambda$ from $t_2$ on, then the same computations developed with $N$ in place of $\tilde{N}_+$ imply, using the first equation in \eqref{DS+}, that $X$ blows up in finite time. Both ways one gets a contradiction.

Let us now prove \eqref{Z< lambdaN} for $\mathcal{M}^+$. Notice that $\dot{Z}<0$ in the region above the line $Z=\lambda (N+a)$ which is contained in $R^+_\lambda$, see Figure \ref{Fig flow}.
Now, if $Z=\lambda (N+a)$ occurs at some point for the orbit $\tau$, then in particular there is some $t_0$ such that $Z>\lambda (N+a)$ for all $t\leq t_0$, thus $\dot{Z}\leq Z(N+a-Z/\lambda)$. In particular $\tau$ remain in the region $R^+_\lambda$ up to the time $t_0$. Hence,
\begin{center}
	$\frac{\lambda (N+a)  \dot{Z} }{\lambda (N+a)-Z}
	=\frac{\dot{Z}}{Z}-\frac{\dot{Z}}{Z-\lambda (N+a)}
	=\frac{\rmd}{\rmd t} \{\, \mathrm{ln} (\frac{Z(t)}{Z(t)-\lambda (N+a)}\, ) \} $ \; for all $t\leq t_0$.
\end{center}
integration in $[t,t_0]$ as before gives us that the trajectory blows up in finite time.

The proof of \eqref{X< tildeN-2} and \eqref{Z< lambdaN} for the operator $\M^-$ is analogous if one uses $\tilde{N}_-\geq N$. 
\end{proof}

Now, a similar argument of Proposition \ref{AP bounds} allows to characterize all the orbits in $3Q$. 

\begin{prop}\label{prop 3Q}
Every orbit of \eqref{DS+ 3Q}  or \eqref{DS- 3Q} in $3Q$ blows up in finite time, backward and forward. 
The vector field in there always point to the right and down, with $\dot{X}>0$ and $\dot{Z}<0$. 
\end{prop}

\begin{proof}
Recall that in $3Q$ we have $X,Z<0$. Let us consider  $\M^+$. Hence, by the first equation in \eqref{DS+ 3Q} one gets $\dot{X}\geq X(X-(\tilde{N}_+-2))$, which is positive.
Similarly, by the second equation in \eqref{DS- 3Q} one figures out that $\dot{Z}\leq Z(\tilde{N}_+ +a-Z/\Lambda)$, which is now negative.
Then integration as in \eqref{int1F X}, \eqref{int2F X} gives us the result. For $\M^-$ it is analogous.
\end{proof}

\section{Classification of solutions}\label{section classification}

In this section we classify the solutions of the second order equations \eqref{P radial m} and \eqref{P radial m-} and we show that this induces a classification of the orbits of the dynamical systems \eqref{DS+}-\eqref{DS- 3Q}.
We investigate three kinds of solutions to \eqref{P radial m} and \eqref{P radial m-}: regular solutions, singular solutions, and exterior domain solutions. The denomination of the latter ones will be clarified in Section \ref{subsection exterior}.

We begin by characterizing the blow ups admissible in the first quadrant.

\begin{prop}[Blow-up types in $1Q$]\label{types of blow ups} 
Let $u$ be a positive solution of \eqref{P radial m} or \eqref{P radial m-} in an interval $(R_1,R_2)$, $0<R_1<R_2$, and $\tau=(X,Z)$ be a corresponding trajectory of \eqref{DS+} or \eqref{DS-} lying in $1Q$ through the transformation \eqref{X,Z a}. Then the following holds:
\begin{enumerate}[(i)]
	\item there exists $r_1\in (R_1,R_2)$ such that $u^\prime (r_1)= 0\Leftrightarrow $ there exists $t_1\in \real$ such that $Z(t)\to +\infty $ as $t\to t_1^+$. In addition, $X(t)\to 0$ as $t\to t_1^+$;
 	
	\item there exists  $r_2\in (R_1,R_2)$ such that $u (r_2)= 0\Leftrightarrow $ there exists $t_2\in \real$ such that $X(t)\to +\infty$ as $t\to t_2^-$. Further, $Z(t)\to 0$ as $t\to t_2^-$.
\end{enumerate}
Moreover, no other blow-up types other than those of $(i)$ and $(ii)$ are admissible for $\tau$ in $1Q$.
\end{prop}

\begin{proof}
Let us first observe that $u$ and $u'$ can never be zero at the same point $r_1$. Otherwise, by the uniqueness of the initial value problem we would have $u\equiv 0$ in a neighborhood of $r_1$, which is not possible by the strong maximum principle.

$(i)$ Assume that there exists $r_1>0$ such that $u^\prime (r_1)=0$. Thus $u(r_1)>0$ and by \eqref{X,Z a} it is easy to deduce the limits of $X(t)$ and $Z(t)$ as $t\to t_1^+$, for $t_1=\mathrm{ln}(r_1)$.
Viceversa, if $Z(t)\to +\infty$ as $t\to t_1^+$, by \eqref{X,Z a} we immediately get $u^\prime (r)\to 0$ as $r\to r_1=e^{t_1}$, because $u$ is continuous in $(R_1,R_2)$. This in turn gives that $X(t)\to 0$ as $t\to t_1^+$, and no other asymptote parallel to the $Z$ axis is admissible.

$(ii)$ Suppose that $u(r_2)=0$ for some $r_2>R_1$.
Then $u^\prime (r_2)<0$ and by \eqref{X,Z a} we easily obtain the behavior of $X$ and $Z$ as $t\to t_2^-$, where $t_2=\mathrm{ln}(r_2)$.
Viceversa if $X(t) \to +\infty$ as $t\to t_2^-$ then necessarily $u(r)\to 0$ as $r\to r_2=e^{t_2}$, because $u^\prime$ is continuous in $(R_1,R_2)$. Thus $Z(t)\to 0$ as $t\to t_2^-$ as before.

The arguments above also show that, in finite time, no other blow-ups are possible for $\tau$ in $1Q$. Indeed, as soon as $X$ or $Z$ tends to infinity, then $u$ or $u^\prime$ vanishes at a positive radius. Recall that a blow up in infinite time is not admissible by Proposition \ref{AP bounds}.
\end{proof}

\begin{corol}\label{cor change at least once}
Let $u$ be a solution of \eqref{P radial m} or \eqref{P radial m-}, and $\tau$ be a corresponding trajectory of \eqref{DS+} or \eqref{DS-} starting above the line $\ell_\pm$ in $1Q$. Then $u$ changes concavity at least once.
\end{corol}

\begin{proof} Consider the $\M^+$ operator, for $\M^-$ is the same.
If $u$ never changed concavity, then $\tau=(X,Z)$ would remain inside the region $R_\lambda^+$ for all time. By Lemma \ref{stationary M+} and Proposition \ref{Dulac} there are no stationary points or periodic orbits in $R_\lambda^+$.
Recall that $\dot{X}>0$, $\dot{Z}<0$ in $R_\lambda^+$, see Figure \ref{Fig flow}.
Then $\tau$ must blow up at a finite forward time $\hat{t}$ such that $X(t)\to +\infty $ and $Z(t)\to Z_1$ as $t\to \hat{t}$, for some $Z_1>\lambda (N-1)>0$.
But this blow-up is not admissible by Proposition \ref{types of blow ups}.  
\end{proof}

\begin{rmk}[Blow up in $3Q$]
	Every orbit ${\tau}=(X,Z)$ of \eqref{DS+ 3Q} or \eqref{DS- 3Q} in $3Q$ verifies $X(t)\to 0$ and $Z(t)\to -\infty$ as $t\to t_1^-$ for some $t_1\in \real$ such that $r_1=e^{t_1}$ and $u^\prime (r_1)= 0$. Moreover, $X(t)\to -\infty$ and $Z(t)\to 0$ as $t\to t_3^+$ for some $t_3\in \real$ where $r_3=e^{t_3}$ and $u (r_3)= 0$.
\end{rmk}

\subsection{Regular solutions}\label{subsection regular}

Let us consider the following initial value problem:
\begin{align}\label{shooting}
u^{\prime\prime}\;=\; M_\pm\left( -r^{-1}(N-1)\, m_\pm(u^\prime)- r^a |u|^{p-1}u \right),
\quad u (0)=\gamma  , \;\; u^\prime (0)=0 , \qquad \gamma>0 ,
\end{align}
where $M_\pm$ and $m_\pm$ are defined in \eqref{m,M+}, \eqref{m,M-}.

By regular solution we mean a solution $u=u_p$ of \eqref{shooting} which is twice differentiable for $r>0$, and $C^1$ up to $0$.  We denote by $R_p$, with $R_p\leq +\infty$, the radius of the maximal interval $[0,R_p)$ where $u$ is positive.
 
Hence, in such interval $u$ is a solution of $(P_\pm)$. 
Obviously, if $R_p=+\infty$ then $u$ corresponds to a radial positive solution of \eqref{P} for $\Omega=\rN$.
When $R_p<+\infty$ it gives a positive solution of the Dirichlet problem \eqref{P}, \eqref{H Dirichlet} in the ball $\Omega= B_{R_p}$.

\begin{rmk}\label{rescaling}
Given a regular positive solution $u=u_p$ in $[0,R_p)$ satisfying \eqref{shooting} for some $\gamma>0$, then the rescaled function $v(r)=\tau u(\tau^{\frac{1}{\alpha}} r)$, for $\alpha$ as in \eqref{critical exponents a} and $\tau>0$, is still a positive solution of the same equation in $[0, \tau^{-\frac{1}{\alpha}}R_p)$ with initial value $v(0)=\tau \gamma$, see also \cite[Lemma 2.3]{FQaihp}.

If $u$ is defined in the whole interval $[0,+\infty)$, thus there is a family of regular solutions obtained via $v=v_\tau$, for all $\tau>0$. 
In this case we say that $u$ is \textit{unique} up to scaling.
 
On the other hand, a solution in the ball of radius $R_p$ automatically produces a solution for an arbitrary ball, by properly choosing the parameter $\tau>0$.
\end{rmk}

\begin{rmk}
Note that, by rescaling a fast decaying solution, we get infinitely many fast decaying solutions which give a different value for the constant $c$ in Definition \ref{def regular}(i), see \eqref{omega A0}. The same happens for the $(\tilde{N} - 2)$--blowing up solution in Definition \ref{def singular}(i). Instead it is easy to see that for the slow decaying solutions or $\alpha$--blowing up solutions in definitions \ref{def regular}(ii) and \ref{def singular}(ii), the constant $c$ is independent of the rescaling, see \eqref{omega M0}.
\end{rmk}

Now, using the transformation \eqref{X,Z a} our goal is to characterize the regular solutions of \eqref{P radial m} or \eqref{P radial m-} as trajectories of the dynamical systems \eqref{DS+} and \eqref{DS-} in the first quadrant.

\begin{prop}\label{prop regular N0}
Let $u=u_p$ be any positive regular solution of \eqref{P radial m} $($resp.\ \eqref{P radial m-}$)$. Then the corresponding trajectory belongs to $1Q$ and is the unique trajectory of \eqref{DS+} $($resp.\ \eqref{DS-}$)$ whose $\alpha$-limit is $N_0$.
\end{prop}

\begin{proof}
The proof is the same for both operators $\M^\pm$. 
The solution $u=u_p$ satisfies $\lim_{r\to 0} u(t)=u(0)=\gamma $ and $\lim_{r\to 0} u^\prime (t)= u^\prime (0)=0$, for some $\gamma>0$.
In terms of the trajectory $(X,Z)$, by the definition of $X$ in \eqref{X,Z a} we easily find
\begin{align}\label{regular X to 0}
\textstyle{\lim_{t\rightarrow -\infty}X(t)= 0.}
\end{align}
Moreover, in the simpler case when $a=0$ we have
\begin{center}
	$\lim_{t\to -\infty}Z(t)=
	\lim_{r\rightarrow 0}\frac{-r  u^{p}(r)}{u^\prime (r)} =
	- \gamma^p  \lim_{r\rightarrow 0}\,\frac{r }{u^\prime (r)-u^\prime (0)} =-\frac{ \gamma^p}{u^{\prime\prime}(0)} \in (0,+\infty)$,
\end{center}
since it is easier to check from the equation that $u^{\prime\prime}(0)<0$.
When $a\neq 0$ we need some other argument to show that $Z(t)$ has a finite limit as $t\to -\infty$.
First let us show that
\begin{align}\label{3.4F}
\textrm{there exists $R_1>0$ such that } u^\prime \neq 0 \textrm{ for all } r\in (0,R_1).
\end{align}

If this was not true, then there would exist a sequence of positive radii $r_n\to 0$ such that $u^\prime (r_n)=0$.
By the mean value theorem this yields the existence a sequence $s_n\in (r_n,r_{n+1})$ such that $u^{\prime\prime}(s_n)=0$.
Thus, since $u^\prime$ cannot be identically zero in a neighborhood of $0$ by the equation \eqref{P radial m}, then $u$ changes infinitely many times its concavity in a neighborhood of $0$.

In terms of the dynamical system, say \eqref{DS+} for $\M^+$, this means that the respective trajectory intersects the line $\ell_+$ (see \eqref{ell}) more than once as $t\to -\infty$. In particular it should pass from $R^+_\lambda$ to $R^-_\lambda$  infinitely many times, which, by Proposition \ref{prop flow} (1), may only occur at
$X(s_n)>\frac{1+a}{p}$ for $a>-1$. This contradicts the fact that $X(s_n)\to 0$ for large $n$ from \eqref{regular X to 0}. 

By \eqref{3.4F} we have that $Z(t)$ is well defined in some interval $(-\infty,\hat{t})$ so that \eqref{Z< lambdaN} in Proposition \ref{AP bounds} yields $Z(t)<\lambda (N+a)$ for all $t<\hat{t}$. Hence $\lim_{t \to -\infty}Z(t)<+\infty$. 
Note that the trajectory cannot belong to $3Q$ when blow-up in finite time occurs both backward and forward, by Proposition \ref{prop 3Q}.
Moreover, it cannot converge to $O$ by Propositions \ref{prop flow} (2) and \ref{local study stationary points}(1). Indeed, the unstable manifold at $O$ is on the $Z$ axis which cannot correspond to the solution $u$ in any interval $(0,r)$.
Hence it converges to $N_0$, independently of the initial datum $\gamma>0$.
\end{proof}

\begin{rmk}\label{remark Gammap regular}
Thus, by Propositions \ref{prop regular N0} and \ref{local uniqueness} one concludes that a regular solution $u_p$,  corresponds to the unique trajectory $\Gamma_p$ labeled in \eqref{Gammap}, for all $p>1$.
Here, $\Gamma_p$ is defined in a maximal interval $[0,T_p)$, $T_p=\mathrm{ln} R_p\leq +\infty$.
	
Note that the fact that $\Gamma_p$ does not depend on the initial datum of $u_p$ is not a surprise since we already observed that the change of initial datum is equivalent to rescaling the radius. This, in turn, is equivalent to shifting the time for the systems \eqref{DS+}, \eqref{DS-}, which does not produce any change in the trajectory since the system is autonomous.
\end{rmk}

We now prove the monotonicity and concavity properties of the regular solutions, deriving them directly by the dynamical systems \eqref{DS+} or  \eqref{DS-}, and not from the second order ODEs.

\begin{prop}\label{prop decreasing concave}
All regular solutions $u$ of \eqref{P radial m} or \eqref{P radial m-} are concave in an interval $(0,r_0)$ for some $r_0>0$ and change concavity at least once. Moreover, they are decreasing as long as they remain positive. Further, $u^\prime (r)=O(-r^{1+a})$ and $u^{\prime\prime}(r)=O(-r^a) $ as $r\to 0$, for $a>-1$.
\end{prop}

\begin{proof} Let us consider the $\M^+$ case; for $\M^-$ is the same. 	
By Proposition \ref{prop regular N0} the corresponding trajectory $\Gamma_p$ starts at $-\infty$ from the stationary point $N_0$ and enters the region $R_\lambda^+$ which is above the concavity line $\ell_+$; see Proposition \ref{local study stationary points} (item 2) and Remark \ref{graph}.
Then we immediately deduce that $u$ is concave near $r=0$ and changes concavity at least once; see Corollary \ref{cor change at least once}. Next, since $1Q$ is invariant by the flow and corresponds to positive decreasing solutions of \eqref{P radial m} we get the monotonicity claim.
For the asymptotics one computes
\begin{center}
	$- \gamma^p  \lim_{r\rightarrow 0}\,\frac{r^{1+a}  }{u^{\prime}(r)}=	\lim_{t\rightarrow -\infty}Z(t)
	=\lambda(N + a) $, from which
	$\lim_{r\rightarrow 0} \frac{u^{\prime\prime}}{r^a}
	= -\frac{\gamma^p}{\lambda} \frac{a+1}{N+a}$. 
\end{center}
\vspace{-0.6cm}
\end{proof}

Given an exponent $p>1$, for a regular solution $u_p$ of \eqref{P radial m} or \eqref{P radial m-}, 
which is positive in $[0,+\infty)$, Definition \ref{def regular} holds according to its behavior at infinity.
Taking into account that $u_p$ is unique, up to rescaling, as in \cite{FQaihp} we define the following sets:
\begin{align}\label{def F,S,P}
\F \,&=\,\{\, p>1 :\; u_p \;\textrm{ is fast decaying}\, \};\nonumber\\
\Sl \,&=\,\{\, p>1 :\; u_p \;\textrm{ is slow decaying}\, \};\\
\Ps \,&=\,\{\, p>1 :\; u_p \;\textrm{ is pseudo-slow decaying}\, \}.\nonumber
\end{align}
Then we add the set
\vspace{-0.16cm}
\begin{align}\label{def C}
\C\,=\,\{\, p>1 :\; \textrm{$(P_\pm)$ has a solution $u_p$ with $u_p(R_p)=0$\,} \, \},
\end{align}
where, as before, $R_p$ is the radius of the maximal positivity interval for $u_p$.
We characterize the previous sets in terms of the orbits of the dynamical systems \eqref{DS+} or \eqref{DS-}.

\begin{prop}\label{prop classifi regular}
In terms of the dynamical systems \eqref{DS+} or \eqref{DS-}, the previous sets can be equivalently defined as follows:
\begin{align}\label{def F,S,P,C DS}
\F \,&=\,\{\, p>1  :\; \omega(\Gamma_p)=A_0\, \};\quad
\Sl \,=\,\{\, p>1 : \;  \; \omega(\Gamma_p)=M_0 \, \};\nonumber\\
\Ps \,&=\,\{\, p>1 : \;  \;\textrm{$\omega(\Gamma_p)$ is a periodic orbit around $M_0$}  \, \};\\
\C\,&=\,\{\, p>1: \; \textrm{$\textstyle{\lim_{t\to T} X(t)=+\infty}$ and  $ \textstyle{\lim_{t\to T} Z(t)=0}$ for some $T>0$}\, \},\nonumber
\end{align}
where $\Gamma_p(t)=(X(t),Z(t))$ is as in Remark \ref{remark Gammap regular}. 
In particular, $(1,+\infty)=\mathcal{C}\cup\mathcal{F}\cup\mathcal{P}\cup\mathcal{S}$.
\end{prop} 

\begin{proof}
The proof is the same for both operators $\M^\pm$, i.e.\ for both systems \eqref{DS+} and \eqref{DS-}.

In the case of the sets $\F,\Sl$ and $\Ps$, $u_p$ (as in Proposition \ref{prop regular N0}) is positive in $(0,+\infty)$ which implies that $\Gamma_p$ is defined for all $t\in\real$. By Proposition \ref{AP bounds} this trajectory is bounded, and so by Poincaré-Bendixson theorem it converges as $r\to+\infty$ either to a stationary point or to a periodic orbit.
In the first case only $A_0$ and $M_0$ are admissible. Moreover, via the transformation \eqref{X,Z a},
\begin{align}\label{omega A0}
\omega (\Gamma_p)=A_0 \;\;\Leftrightarrow\;\; \lim_{r\to +\infty}u(r)r^{\tilde{N}_\pm-2}=C_{p,\gamma} \quad\textrm{ for some $C_{p,\gamma} >0$},
\end{align}
and $u=u_p$ has fast decay at $+\infty$. On the other hand $u$ is slow decaying at $+\infty$ when
\begin{align}\label{omega M0}
\omega (\Gamma_p)=M_0 \;\;\Leftrightarrow\;\; \lim_{r\to +\infty}u(r)r^\alpha=C_p \quad\textrm{ with $C_p =(X_0 Z_0)^{\frac{1}{p-1}}$},
\end{align}
where $M_0=(X_0,Z_0)$ is given explicitly in Lemma \ref{stationary M+}.

\smallskip

Indeed, \eqref{omega M0} comes from the identity $X(t)Z(t)=r^{2+a} u^{p-1}$ for all $t\in \real$.
On the other hand,
\begin{center}
$\lim_{t \to +\infty} X(t)= \tilde{N}_\pm -2$\;\, is equivalent to\; 
$\frac{\rmd}{\rmd r} \mathrm{ln}(u(r)) = \frac{u^\prime(r)}{u(r)} \sim -\frac{\tilde{N}_\pm-2}{r}$ as $r\to +\infty$.
\end{center}
Then, integration in $[r_0,r]$ for a fixed large $r_0$ implies \eqref{omega A0} with $C_{p,\gamma}=u(r_0)r_0^{\tilde{N}_\pm-2}$, $u=u_{p,\gamma}$.
Now, by rescaling, the function $v=v_\tau$ in Remark \ref{rescaling} satisfies 
\begin{center}
$\lim_{t\to +\infty}v(r) r^{\tilde{N}_\pm-2}=\tau^{1-\frac{\tilde{N}_\pm-2}{\alpha}}C_{p,\gamma}$ under \eqref{omega A0};\;\, $\lim_{r\to +\infty}v(r)r^\alpha=C_p$ under \eqref{omega M0}.
\end{center}
Thus $C_p$ is independent of the initial value $\gamma>0$ in \eqref{shooting}.

Finally, assume that $\omega(\Gamma_p)$ is a periodic orbit $\theta$. Note that the region inside $\theta$ is bounded, and by Poincaré-Bendixson theorem it must contain $M_0$. Using again $XZ=r^{2+a}u^{p-1}$ one defines
\begin{center}
$c_1^{p-1}:=\inf_{t\in \real} \,\{ X(t)Z(t): (X,Z)\in \theta\}$; \;\; $c_2^{p-1}=\sup_{t\in \real} \,\{ X(t)Z(t): (X,Z)\in \theta\}$.
\end{center}
Therefore we deduce 
\begin{align}\label{omega theta}
\textstyle{\omega (\Gamma_p)=\theta \;\;\Leftrightarrow\;\; 0<c_1=\liminf_{r\to +\infty}u(r)r^\alpha<\limsup_{r\to +\infty}u(r)r^\alpha=c_2.}
\end{align}

\smallskip

Now we consider the set $\C$. The corresponding trajectory $\Gamma_p$ cannot be defined for all time since $u(R_p)=0$. So it must blow up at the finite time $T_p=\mathrm{ln}(R_p)$ by Proposition \ref{types of blow ups}.

Viceversa if $\F,\Sl,\Ps,\C$ are defined in terms of the property of the trajectory $\Gamma_p$ of the dynamical system then they give exactly the same sets as defined for $u_p$, by using \eqref{X,Z a} and \eqref{def u via X,Z}, and arguing in a similar way.
\end{proof}

\begin{rmk}[$\C$ is open]\label{C is open}
	When $p\in \C$, the trajectory $\Gamma_p$ crosses the line $L=\{(X,Z), X=\tilde{N}_+-2\}$ and next blows up in finite time. This property is preserved for $p^\prime$ close to $p$. 
\end{rmk}

\subsection{Singular solutions}\label{subsection singular}

As mentioned in Section \ref{Introduction}, by singular solution we mean a radial solution $u$ of \eqref{P} satisfying \eqref{H singular}. Hence $u=u(|x|)=u(r)$ is a singular solution of $(P_\pm)$ satisfying $\lim_{r\to 0} u(r) = +\infty$.
It may be either positive for all $r\in (0,+\infty)$, or be equal to zero at a certain radius $R>0$. 
In the latter case it produces a solution in $B_R\setminus\{0\}$.

In terms of the systems \eqref{DS+}, \eqref{DS-}, this means that the corresponding trajectory, say $\Sigma_p$, will be defined either in $\real$ or in an interval $(-\infty,T)$, for some $T<\infty$.
Under the latter, as in the characterization of $\C$ in \eqref{def F,S,P,C DS} we have that $\Sigma_p$ blows up forward in a finite time $T<+\infty$.
Otherwise, by Proposition \ref{AP bounds} the global trajectory $\Sigma_p$ is contained in the box
\begin{center}
 $\mathcal{Q}^+=(0,\tilde{N}_+-2)\times (0,\lambda (N+a))$ for $\M^+$; \;
$\mathcal{Q}^-=(0,\tilde{N}_- -2)\times (0,\Lambda (N+a))$ for $\M^-$.
\end{center}
Then the $\alpha$ and $\omega$ limits can be either a periodic orbit or a stationary point.

We point out that $\Sigma_p$ cannot converge to $N_0$, neither backward nor in forward time, because the stable direction at $N_0$ is the $Z$ axis, while the unstable direction corresponds to the regular trajectory $\Gamma_p$, for all $p>1$. So all possible $\alpha$ and $\omega$ limits of $\Sigma_p$ are $M_0,A_0$, or a periodic orbit.

By the analysis of the stationary points $M_0$ and $A_0$, and of the periodic orbits given in Section~\ref{section periodic}, the $\alpha$ and $\omega$ limits of $\Sigma_p$ depend on the exponent $p$. 
Then a classification of the singular solutions, according to Definition \ref{def singular} can be easily formulated in terms of the dynamical systems \eqref{DS+},  \eqref{DS-}, analogously to Proposition \ref{prop classifi regular}. Obviously if $\Sigma_p$ is defined in $\real$ they are also classified according to the behavior at $+\infty$, as in Definition \ref{def regular}.
Here we just emphasize that, as for the regular solutions, the so called pseudo--blowing up solutions, see Definition \ref{def singular} (iii), may only occur at the values of $p$ for which $\Sigma_p$ has a periodic orbit as $\alpha$-limit.

\subsection{Exterior domain solutions}\label{subsection exterior}

By exterior domain solution we mean a solution $u$ of \eqref{P radial m} or \eqref{P radial m-} defined in an interval $[\rho_0,\rho]$, for $\rho_0\in (0, +\infty)$ and $\rho\leq +\infty$, and verifying the Dirichlet condition $u(\rho_0)=0$. Fixing $\rho_0=1$, then $u$ satisfies the initial value problem
\begin{align}\label{shooting derivative}
u(1)=0, \qquad u^\prime (1)=\delta, \qquad \textrm{ for some }\; \delta >0.
\end{align}

The equations \eqref{P radial m}, \eqref{P radial m-}, together with \eqref{shooting derivative} were studied in \cite{GILexterior2019} and \cite{GLPradial}. 
It was shown that for any $p>1$ and for each $\delta>0$ there exists a unique solution $u=u_\delta$ defined in a maximal interval $(1, \rho_\delta)$ where $u$ is positive, $1<\rho_\delta\leq +\infty$.
Moreover, there exists a unique $\mu_\delta\in (1,\rho_\delta)$ such that
\begin{align}\label{mu delta}
	u^\prime_\delta(r)>0 \;\;\; \hbox{for } r< \mu_\delta\, ,\qquad u^\prime_\delta(\mu_\delta )=0\, ,\qquad  u^\prime_\delta(r)<0 \;\;\; \hbox{for } r>\mu_\delta\, .
\end{align}
If $\rho_\delta=+\infty$ we get a positive radial solution in the exterior of the ball $B_1$. In the second case, a positive solution in the annulus $(1,\rho_\delta)$ is produced.
Note that equations \eqref{P radial m}, \eqref{P radial m-} together with \eqref{shooting derivative} are not invariant by rescaling. 

Now we would like to describe the trajectories of the dynamical systems \eqref{DS+}-\eqref{DS-} which correspond to $u_\delta$ through the variables $X,Z$ in \eqref{X,Z a}.

\begin{prop}\label{prop exterior}
Let $p>1$ and $u_\delta=u_{\delta,p}$ be a positive solution of \eqref{P radial m} $($resp.\ \eqref{P radial m-}$)$ satisfying \eqref{shooting derivative}. Then there exists a unique trajectory $\varXi_{\delta,p}$ in $1Q$ for the system \eqref{DS+} $($resp.\ \eqref{DS-}$)$ which blows up backward in a finite time $t_\delta$. More precisely, if $\varXi_{\delta,p}(t)=(X(t),Z(t))$ then 
\begin{align}\label{blow up exterior 1Q}
\textstyle{ \lim_{t\to t_\delta^+} Z(t)=+\infty \textrm{ \;and \;} \lim_{t\to t_\delta^+} X(t)=0 },
\end{align} 
where $t_\delta=\mathrm{ln}(\mu_\delta)$, $\mu_\delta$ given in \eqref{mu delta}. The trajectory $\varXi_{\delta,p}$ corresponds, after the transformation \eqref{X,Z a} to the restriction of $u_\delta$ to the interval $I_\delta=(\mu_\delta,\rho_\delta)$, $\rho_\delta\leq +\infty$.
\end{prop}

\begin{proof}
The proof works for both operators $\M^\pm$. 
We fix $p$ and $\delta$. By \eqref{mu delta}, $u_\delta$ is positive and decreasing in the interval $I_\delta$. Then, after \eqref{X,Z a}, to $u_\delta$ corresponds a unique trajectory $\varXi_\delta =\varXi_{\delta ,p}$ contained to $1Q$ and is defined for all $t\in (t_\delta, \mathrm{ln}(\rho_\delta))$. Thus, by Proposition \ref{types of blow ups} we get \eqref{blow up exterior 1Q}.
\end{proof}

When $\rho_\delta=+\infty$ we can classify the solutions accordingly to their behavior at $+\infty$, i.e.\ $u_\delta$ is fast, slow, or pseudo-slow decaying via the limits (i)-(iii) as $r\to +\infty$ in Definition \ref{def regular}.

\section{The $\mathcal{M}^+$ case}\label{section Pucci+}

In this section we study the solutions of the equations involving the Pucci $\M^+$ operator. 
Hence we refer to the dimension-like parameter $\tilde{N}_+$ and the relevant exponents for $\M^+$ defined in \eqref{def dimensional-like} and \eqref{critical exponents a}, as well as their ordering:
\,
$
\max \{\,p^{s,a}_+ ,\, p^a_\Delta\} \leq p_+^{p,a} .
$

\subsection{Some properties of regular trajectories}\label{section crossing}
We first consider the case of a regular solution of \eqref{P radial m} whose corresponding trajectory for the system \eqref{DS+} will be denoted by $\Gamma_p=\Gamma_p(t)$ as in Section \ref{subsection regular}. We also keep the other notations already introduced, in particular for the sets $\F,\Sl, \Ps, \C$ defined in \eqref{def F,S,P}, \eqref{def C}, and Proposition \ref{prop classifi regular}.

\begin{lem}\label{lem l1, l2 for X>alpha}
For any $p>1$, with $\Gamma_p=(X_p,Z_p)$, we have:
\begin{enumerate}[(i)]
\item if $\Gamma_p$ reaches the line $\ell_1^+$ $($see \eqref{ell 1+}$)$ at some $t_0$ with $X_p(t_0)\geq \alpha$, then $p\in \Sl \cup \Ps$, i.e.\ the corresponding solutions $u_p$ of \eqref{P radial m} are either slow decaying or pseudo-slow decaying. In the latter case $\Gamma_p$ crosses $\ell_1^+$ and $\ell_2^+$ $($see \eqref{ell 2+}$)$ infinitely many times;
	
\item if $\Gamma_p$ does not intersect the line $\ell_1^+$, then it intersects the concavity line $\ell_+$ exactly once. Moreover, $\dot{X}_p>0$ and $\dot{Z}_p< 0$ for all time. In particular this happens for $p\in \mathcal{F}\cup \mathcal{C}$.
\end{enumerate}
\end{lem}

\begin{proof}
We recall that $\Gamma_p$ starts at $-\infty$ from the stationary point $N_0$ and must cross the concavity line $\ell_+$ at least once, see Proposition \ref{prop decreasing concave}.  

$(i)$ If $\Gamma_p$ reaches $\ell_1^+$ for $X_p(t_0)= \alpha$ then clearly $\lim_{t\to +\infty} \Gamma_p (t)=M_0$, whenever $M_0$ belongs to $1Q$ (see Fig.~\ref{Fig flow}). If instead $X_p(t_0)>\alpha$, by taking into account Proposition \ref{prop flow} (3) (see again Fig.~\ref{Fig flow}) and that $\Gamma_p$ cannot self intersect, we have that $\Gamma_p$ is contained in a bounded region from which it cannot leave.
Thus, by Poincaré-Bendixson theorem the $\omega$-limit of $\Gamma_p$ is either $M_0$ or a periodic orbit $\theta$ which contains $M_0$ in its interior. In the latter case $\Gamma_p$ goes around $\theta$ clockwisely according to the direction of the vector field, intersecting $\ell_1^+$ and $\ell_2^+$ infinitely many times.

$(ii)$ If $\Gamma_p$ does not intersect $\ell_1^+$ then it cannot turn back and cross the concavity line $\ell_+$ another time because of the direction of the flow.	
Moreover, it can neither intersect nor be tangent to the line $\ell_2^+$ where $\dot{Z}=0$, since a $C^1$ trajectory of \eqref{DS+} may only intersect the line $\ell^+_2$ transversely by passing from left to right, see Proposition \ref{prop flow} (3) and Fig.~\ref{Fig flow}.
This fact and item $(i)$ conclude the final assertion. 
\end{proof}

The next proposition is crucial to study the behavior of $\Gamma_p$ for different values of $p$.

\begin{prop}\label{prop crossing Gammap}
Assume that $p_1\in \mathcal{F}\cup \mathcal{C}$, and let $\Gamma_{p_2}$ be any regular trajectory with $p_2\neq p_1$. Then $\Gamma_{p_1}$ and $\Gamma_{p_2}$ can never intersect.
\end{prop}

\begin{proof} 
Both $\Gamma_{p_1}$ and $\Gamma_{p_2}$ have their $\alpha$-limit at the stationary point $N_0$ which is a saddle point. By Proposition \ref{local study stationary points}(2) the tangent unstable directions for $\Gamma_{p_1}$ and $\Gamma_{p_2}$ at $N_0$ are given, respectively, by
\begin{align}\label{directions p1,p2}
\textstyle{Z=-\frac{p_1\lambda (N+a)}{N+2+2a}X}\quad \textrm{ and }\quad \textstyle{Z=-\frac{p_2\lambda (N+a)}{N+2+2a}X}.
\end{align} 
Assume by contradiction that $\Gamma_{p_1}(t)=(X_1(t),Z_1(t))$ and $\Gamma_{p_2}(t)=(X_2(t),Z_2(t))$ intersect. Let us denote by $Q$ the first intersection point. Since the dynamical system \eqref{DS+} is autonomous, one may assume that the intersection happens at the same time for both trajectories, i.e.\ $Q= (X_1(t_0),Z_1(t_0))=(X_2(t_0),Z_2(t_0))$.

To fix the ideas assume $p_1<p_2$. Then, by \eqref{directions p1,p2}, at least in a neighborhood of $N_0$, $\Gamma_{p_1}$ is above $\Gamma_{p_2}$ because $X\to 0^+$ (from the right). Moreover, from \eqref{DS+} and Lemma \ref{lem l1, l2 for X>alpha}(ii) we have
\begin{align}\label{derivatives p1<p2}
\dot{X}_1(t_0)=\dot{X}_2(t_0)>0, \quad \dot{Z}_2(t_0)<\dot{Z}_1(t_0)< 0,
\end{align}
since only $\dot{Z}$ depends on $p$.
In particular $\Gamma_{p_1}$ remains above $\Gamma_{p_2}$ after intersecting. Thus the two trajectories must have the same tangent at the point $Q$, which is not possible by \eqref{derivatives p1<p2}.
The case $p_2<p_1$ is analogous.
\end{proof}

From the previous results we immediately get that a fast decaying solution can exist for only one value of $p$.

\begin{corol}\label{prop at most one F}
	There exists at most one $p$ in the interval $(p_+^{s,a},+\infty)$ such that $p\in \F$.
\end{corol}

\begin{proof}
Assume by contradiction that $p_1,p_2\in \F$ for some $p_+^{s,a}<p_1<p_2$.
This means that the corresponding trajectories $\Gamma_{p_1}$ and $\Gamma_{p_2}$ both come out from $N_0$ at $-\infty$ and converge to $A_0$ at $+\infty$.
We have already observed by \eqref{directions p1,p2} that $\Gamma_{p_1}$ stays above $\Gamma_{p_2}$ in a neighborhood of $N_0$. 

On the other hand, since $A_0$ is a saddle point for $p>p_+^{s,a}$, looking at the linear stable directions given by Proposition \ref{local study stationary points}(3), we have that $\Gamma_{p_1}$ and $\Gamma_{p_2}$ arrive at $A_0$ with a reversed order; i.e.\ $\Gamma_{p_2}$ is above $\Gamma_{p_1}$. This is because $X\to (\tilde{N}_+-2)$ from the left.

Hence, $\Gamma_{p_1}$ and $\Gamma_{p_2}$ should intersect, but this is not possible by Proposition \ref{prop crossing Gammap}.
\end{proof}

Another important consequence of Proposition \ref{prop crossing Gammap} is the following result.

\begin{corol}\label{cor separation p0}
Let $p_0\in \F$, $p_0>p_+^{s,a}$, then $p\in \C$ for $p_+^{s,a}<p<p_0$, and $p\in\Ps\cup\Sl$ for $p>p_0$.
\end{corol}

\begin{proof}
If $p_0 \in \F$ then $\Gamma_{p_0}$ cannot intersect any other regular orbit $\Gamma_p$ for $p\neq p_0$, by Proposition \ref{prop crossing Gammap}. This means that if $\Gamma_p$ is above or below $\Gamma_{p_0}$ in a neighborhood of $N_0$, it remains so for all time. Moreover, $p \not\in \F$ for $p\neq p_0$ by Corollary \ref{prop at most one F}.

Thus, if $p<p_0$, $\Gamma_p$ lies above $\Gamma_{p_0}$ and so cannot converge to $M_0=M_0(p)$ neither to a periodic orbit around it, since the line $\ell_1^+$ is below $\Gamma_{p_0}$. Notice that $\ell_1^+$ does not depend on $p$. So $p\in\C$. 

On the other hand, if $p>p_0$ then $\Gamma_p$ lies below $\Gamma_{p_0}$ and therefore cannot cross the line $L=\{(X,Z): X = \tilde{N}_+ -2\}$ in order to blow up in finite time.
Hence $p \not\in \C$ and so must be in $\Ps \cup \Sl$. 
\end{proof}

\subsection{The critical exponent}\label{section critical exponent}

Our goal here is to define and characterize the critical exponent which will be proved to have all properties listed in Theorem \ref{Th M+}.

We start by showing that $\Sl$ and $\C$ contain the intervals $(p^{p,a}_+,+\infty)$ and $(1,p_\Delta)$ respectively.

\begin{prop}\label{p>p^p+}
If $p> p^{p,a}_+$ then $p\in \mathcal{S}$.
\end{prop}

\begin{proof}
In case $p>p^{p,a}_+$, by Proposition \ref{Dulac} we know that there are no periodic orbits of the system \eqref{DS+}, hence $p\not\in \Ps$. Moreover, $M_0$ is a sink by Proposition \ref{local study stationary points}. Let us show that $p\not\in\C\cup \F$.

If $p\in \mathcal{C}$, then $\Gamma_p$ crosses the line $L:=\{(X,Z),X=\tilde{N}_+-2\}$, and blows up in finite time.
Then the region $D$ enclosed by $\Gamma_p$, $L$ and the $X,Z$ axes form a bounded domain from which an orbit can only leave $D$ forward in time through $L$. 
Thus, an orbit arriving at $M_0\in D$ cannot go anywhere in backward time, giving a contradiction with Poincaré-Bendixon theorem.

If instead $p\in \F$, then the bounded set whose boundary is given by $\Gamma_p$ together with the $X$ and $Z$ axes, is invariant and contains $M_0$. Again the orbits arriving at $M_0$ cannot escape in backward time.
Therefore $p\in \Sl$.
\end{proof}

\begin{prop}\label{p<p_Delta}
For $p\in (1,\max\{p^{s,a}_+,p_\Delta^a\})$ it holds that $p\in \mathcal{C}$.
\end{prop}

\begin{proof}
If $1<p<p^{s,a}_+$, then $A_0$ is a source and $M_0\not\in 1Q$. In particular there are no periodic orbits contained in $1Q$. Hence $p\not\in \F\cup\Sl\cup\Ps$, so if $\max\{p^{s,a}_+,p_\Delta^a\}=p^{s,a}_+$ the proof is complete.

Assume $p^{s,a}_+< p_\Delta^a$. Then, at $p=p^{s,a}_+$ no periodic orbits are allowed by Proposition \ref{Dulac},
whose proof also shows nonexistence of homoclinics at $A_0=M_0$ (i.e.\ orbits $\tau$ with $\omega(\tau)=\alpha(\tau)=A_0$).
Therefore, if we had $\omega(\Gamma_p)=A_0$ then the orbits which come out from $A_0$ could not go anywhere. 
Alternatively, see Proposition \ref{p=p^s}.
In particular, $p^{s,a}_+\not\in \F\cup\Sl\cup\Ps$. 

Finally, if $p^{s,a}_+<p<p_\Delta^a$, then $M_0$ is a source by Proposition \ref{local study stationary points}. The trajectory $\Gamma_p$ cannot be bounded, otherwise it could only converge to $A_0$ as $t\to+\infty$. As in the proof of Proposition \ref{p>p^p+} this would produce a contradiction, because the region $D$ enclosed by $\Gamma_p$, and the $X$, $Z$ axes would be an invariant set from which any trajectory issued from $M_0$ cannot exit.

In any case $\Gamma_p$ blows up in finite time, and so $p\in \C$.
\end{proof}

By Propositions \ref{p<p_Delta} and \ref{p>p^p+} we have that the set $\C$ is nonnempty and bounded from above. Therefore we define
\vspace{-0.5cm}
\begin{align}\label{def p*+}
p^*_{a+}=\sup \C
\end{align}
and obviously $p_{a+}^*\in[\max\{p^{s,a}_+,p_\Delta^a\},  p^{p,a}_+]$. From now on we refer to $p^*_{a+}$ as the critical exponent for the Pucci operator $\M^+$ with weight $|x|^a$.
The next result characterizes $p^*_{a+}$.

\begin{theorem}\label{characterization p*}
The number $p^*_{a+}$ defined in \eqref{def p*+} belongs to $\F$. Thus it is the only exponent in the equation \eqref{P radial m} for which there exists a unique, up to scaling, fast decaying solution. 
\vspace{0.07cm}

Moreover, if $\lambda<\Lambda$, then
$p^*_{a+}\neq p_\Delta^a$, $p^*_{a+}\neq p_+^{p,a}$, and \eqref{critical ordering M+} holds.
Further, $\Ps=(p^*_{a+},p^{p,a}_+]$, and for any $p\in \Ps$ the corresponding trajectory $\Gamma_p$ crosses the concavity line $\ell_+$ infinitely many times.
\end{theorem}

\begin{proof}
First, $p^*_{a+}\not\in \C$ i.e.\ $\C$ does not have a maximum because $\C$ is open, see Remark \ref{C is open}. By Proposition \ref{local study stationary points} we know that $M_0$ is a source for every $p\in[\,p_\Delta^a, \, p^{p,a}_+)$; and $M_0$ is a center at $p=p^{p,a}_+$. 
Whence $p\not\in \Sl$ for all $p\in[\,p_\Delta^a, \, p^{p,a}_+]$, and in particular $p^*_{a+}\not\in \Sl$. 
On the other side, if $p^*_{a+}\in \Ps$ then $\Gamma_{p^*_{a+}}$ would cross the line $\ell_1^+$ by Lemma \ref{lem l1, l2 for X>alpha}(ii). Thus, by continuity with respect to $p$, the trajectory $\Gamma_p$ should also cross $\ell_1^+$ for $p$ close to $p^*_{a+}$. But every trajectory $\Gamma_p$ for $p\in \C$ does not cross $\ell_1^+$, by Lemma \ref{lem l1, l2 for X>alpha}. Therefore $p^*_{a+}\not\in \Ps$. 

Hence $p^*_{a+}$ belongs to $\F$ and the trajectory $\Gamma_{p^*_{a+}}$ together with the $X$ and $Z$ axes enclose a bounded invariant region $D$ which contains $M_0$ in its interior. Since $M_0$ is a source for $p\in [p^a_\Delta, p^{p,a}_+)$, and a center for $p=p^{p,a}_+$, the set $D$ contains periodic orbits which cross the line $\ell_+$ twice. Indeed, the flow is subjected to Poincaré-Bendixson theorem, see Figure~\ref{Fig F M+}. This implies that $p^*_{a+}$ can be neither $p^a_\Delta$ nor $p^{p,a}_+$ if $\lambda<\Lambda$, by Proposition \ref{Dulac}. In fact, at $p^a_\Delta$ there are no periodic orbits at all, while at $p^{p,a}_\pm$ no periodic orbits cross $\ell_+$ twice.

Note that we obtain \eqref{critical ordering M+} as long as $p^a_\Delta\geq p^{s,a}_+$. 
If $p^a_\Delta<p^{s,a}_+$ we still need to prove that $p^*_{a+}\neq p^{s,a}_+$. For instance this follows by well known Liouville results in \cite{CutriLeoni}. Alternatively, a proof of this fact is accomplished in Proposition \ref{p=p^s} which give nonexistence of entire positive solutions for $p= p^{s,a}_+$. 	

Next, by Corollary \ref{cor separation p0} we get $(p^*_{a+}, p^{p,a}_+] = \mathcal{P}$, since we have already observed that $[p^a_\Delta, p^{p,a}_+) \cap \Sl =\emptyset$.
By the definition of $\mathcal{P}$, the corresponding trajectory $\Gamma_p$ goes around a periodic orbit $\theta$. By Proposition \ref{Dulac}, if $p<p^{p,a}_+$ then $\theta$ must necessarily intersect both $R^+_\lambda$ and $R^-_\lambda$, while for $p=p^{p,a}_+$ the maximal periodic orbit $\theta_0$ does not intersect $R^+_\lambda$. 

We claim that $\theta_0$ is tangent to $\ell_+$ at the point $P=(\frac{1+a}{p},\lambda (N-1))\in \ell_+\cap \ell_2^+$ when $p=p^{p,a}_+$.
If this was not the case, then $\Gamma=\Gamma_{p^{p,a}_+}$ would belong to the region $R^-_\lambda$ for all $t\in I=[T,+\infty)$ for some $T>0$. Let us consider the restriction of $\Gamma$ to $I$, namely $\tau$. Since $\tau$ is a part of a trajectory for the Laplacian operator in dimension $\tilde{N}_+$, we may follow $\tau$ backward in time as a trajectory of the respective Laplacian-like dynamical system. However, the characterization of $p^{p,a}_+$ as the critical exponent there immediately contradicts the existence of $\tau$. Indeed, at the critical exponent only periodic trajectories are admissible around $M_0$, see for instance the proof of Proposition \ref{prop M0 center}.

Thus, in both cases $\Gamma_p$ for $p\in \Ps$ must cross the concavity line $\ell_+$ infinitely many times. 
\end{proof}

\begin{rmk}[$\Gamma_p=\Upsilon_p$]\label{Gammap=Upsilonp} The critical exponent $p^*_{a+}$ is the unique value of $p$ for which $\Gamma_p$ and $\Upsilon_p$ coincide, see Proposition \ref{local uniqueness}.
\end{rmk}

\begin{proof}[Proof of Theorem \ref{Th M+}]
One establishes the conclusion of Theorem \ref{Th M+} by combining \eqref{def p*+}, Corollary \ref{cor separation p0}, Propositions \ref{p>p^p+} and \ref{p<p_Delta}, together with Theorem \ref{characterization p*}.
\end{proof}

\subsection{Singular and exterior domain solutions}\label{section singular}

Here we show how the analysis of the regular trajectories performed in the previous sections almost completely determines the behavior of the other orbits of the dynamical system \eqref{DS+}.

Let us start by considering singular solutions.
When $p\leq p^{s,a}_+$ we saw in Proposition \ref{p<p_Delta} that $p\in \C$. On the other hand, there is not a unique trajectory arriving at the stationary point $A_0$ as in Proposition \ref{local uniqueness}. 
Indeed, for $p<p^{s,a}_+$, $A_0$ is a source and $M_0$ belongs to the fourth quadrant, see Proposition \ref{local study stationary points}.
The case $p=p^{s,a}_+$ is a bit more involved. The point $A_0=M_0$ is not a hyperbolic point, and we complement its local study in what follows. 

\begin{figure}[!htb]
	\centering
	\includegraphics[scale=0.4]{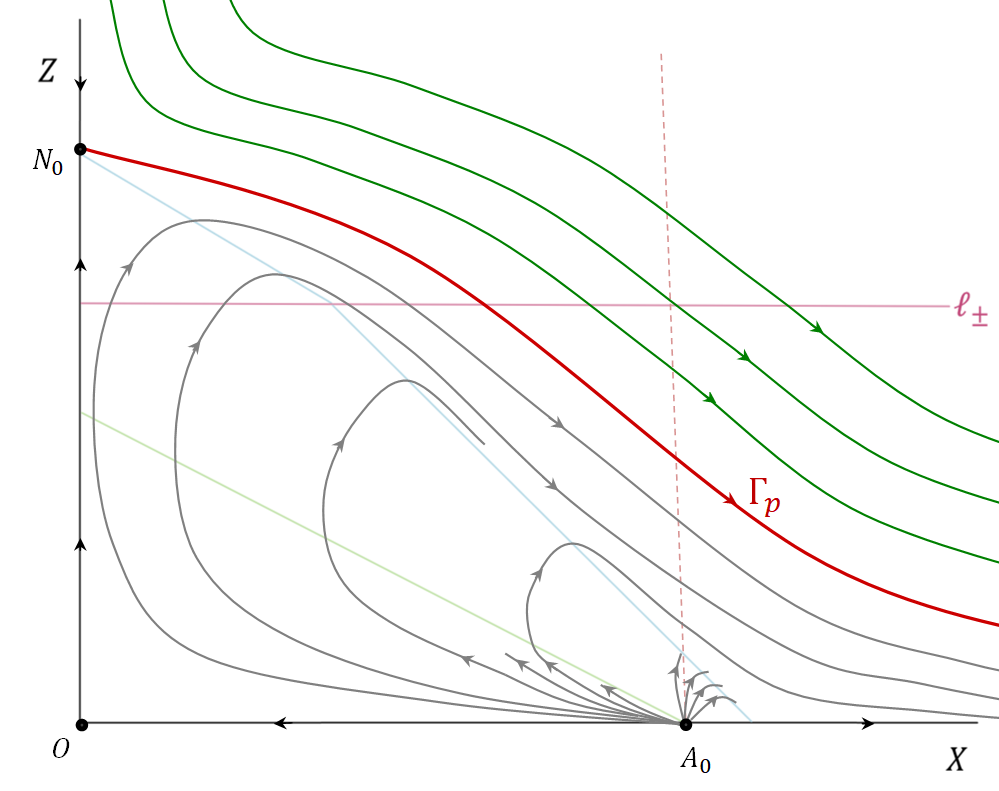}
	\caption{Case $p<p^{s,a}_\pm$ for $\M^\pm$; $A_0$ is a source and $M_0$ belongs to the fourth quadrant. Below $\Gamma_p$ trajectories corresponding to infinitely many $(\tilde{N}_\pm-2)$--blowing up solutions in a ball are shown, while above $\Gamma_p$ there are orbits corresponding to solutions in an annulus.}
	\label{Fig p<ps}
\end{figure}

\begin{prop}\label{p=p^s}
At $p=p^{s,a}_+$,  there exist infinitely many unstable orbits issued from $M_0=A_0$ below the line $\ell_1^+$. They move clockwisely and blow up in finite forward time. 
\end{prop}

\begin{proof}
	At $p=p^{s,a}_+$ the eigenvalues of $A_0=M_0$ are $(\tilde{N}_+-2)$ and $0$. In particular, $A_0$ is not hyperbolic and Proposition~\ref{prop GH} no longer applies. 
	The linear direction corresponding to $(\tilde{N}_+-2)$ lies on the $X$ axis, while the one corresponding to $0$ coincides with the line $\ell_1^+$.
	However, through the flow analysis in Proposition \ref{prop flow} (3) (see Figure \ref{Fig flow2}) it is easy to conclude that $M_0$ has infinitely many repulsive directions between these two lines. In this case, the orbits are issued from $A_0$, with respective tangent lines between the $X$ axis and the line $\ell_1^+$. 
	
	To see this let us first observe that $\ell_1^+$ and $\ell_2^+$ intersect at $A_0$. Then note that $\dot{X}>0$ in the region above $\ell_1^+$. On the other hand, an orbit coming out from $A_0$ needs to increase its $Z$ values, so staying below $\ell_2^+$.
	If it started between $\ell_1^+$ and $\ell_2^+$, then it should initially decrease its $X$ values, which gives a contradiction.
	Hence the only way to come out from $A_0$ is below the line $\ell_1^+$.
	
We have already deduced in the proof of Proposition \ref{p<p_Delta} that periodic orbits at $p^{s,a}_+$ are not admissible if $p^{s,a}_+<p^a_\Delta$ by Dulac's criterion (Proposition \ref{Dulac}). 	However, this is true even if $p^{s,a}\geq p^a_\Delta$ by the flow direction, see Figure \ref{Fig flow2}. Indeed, the region $\dot{X}, \dot{Z}<0$ does not intersect $1Q$.	 By the same reason, the trajectories near $M_0=A_0$ move clockwisely, by intersecting both lines $\ell_1^+$ and $\ell^+_2$ exactly once. 
	 
To conclude we infer that the behavior of the flow on the lines $\ell_1^+$ and $\ell_2^+$ does not allow any orbit to reach $A_0=M_0$ in forward time.
Assume on the contrary that there exists a homoclinic orbit $\tau$ with $\omega(\tau)=\alpha(\tau)=A_0$. In this case $\tau$ creates a bounded invariant region $D$ such that any orbit inside $D$ is also homoclinic, by Poincaré-Bendixson theorem. 
Fix a point $Q_0\in R_\lambda^-\cap \ell_2^+\cap D$, and consider the unique trajectory $\tau_0$ passing through this point at time $t=0$. By construction, $\tau_0$ lies entirely in the region $R_\lambda^-$. However, the proof of Proposition \ref{Dulac}(ii), applied to the region $D_0$ enclosed by the trajectory $\tau_0$, yields a contradiction with the fact that $p^{s,a}_+\neq p^{p,a}_+$. 
\end{proof}

\begin{figure}[!htb]
	\centering
	\includegraphics[scale=0.4]{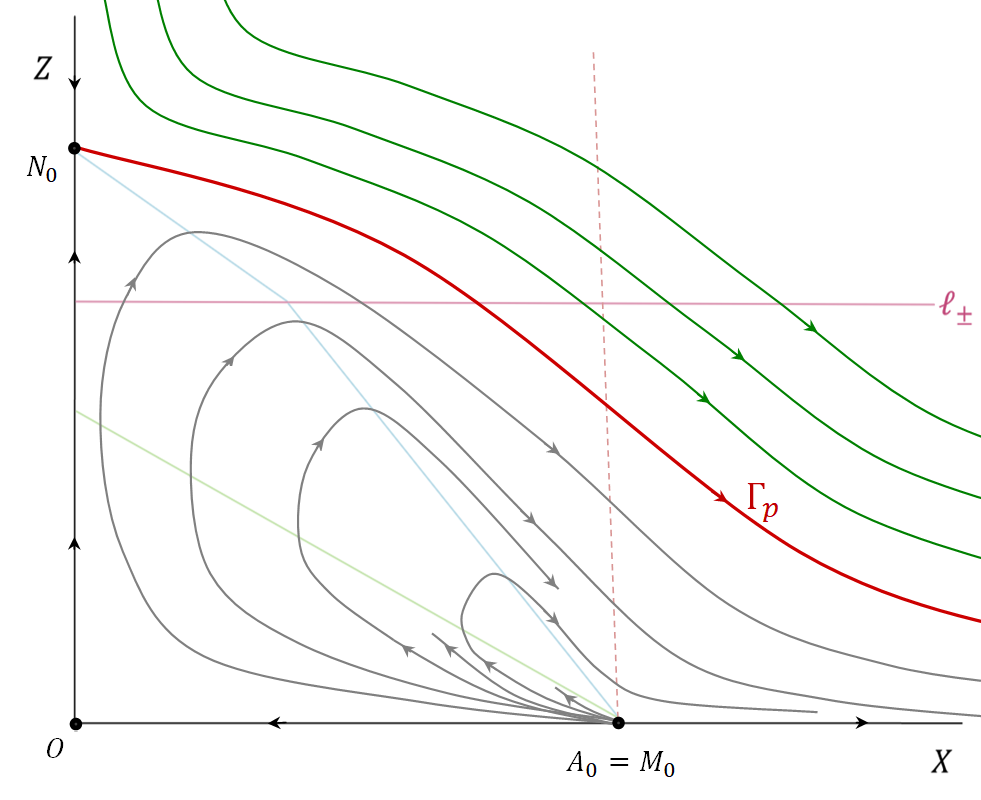}
	\caption{Case $p=p^{s,a}_\pm$: $p\in \C$ and $A_0=M_0$ has infinitely many unstable directions. Below $\Gamma_p$ are the orbits corresponding to infinitely many  $(\tilde{N}_\pm-2)$--blowing up solutions in a ball.}
	\label{Fig p=ps}
\end{figure}

\begin{figure}[!htb]
	\centering
	\includegraphics[scale=0.44]{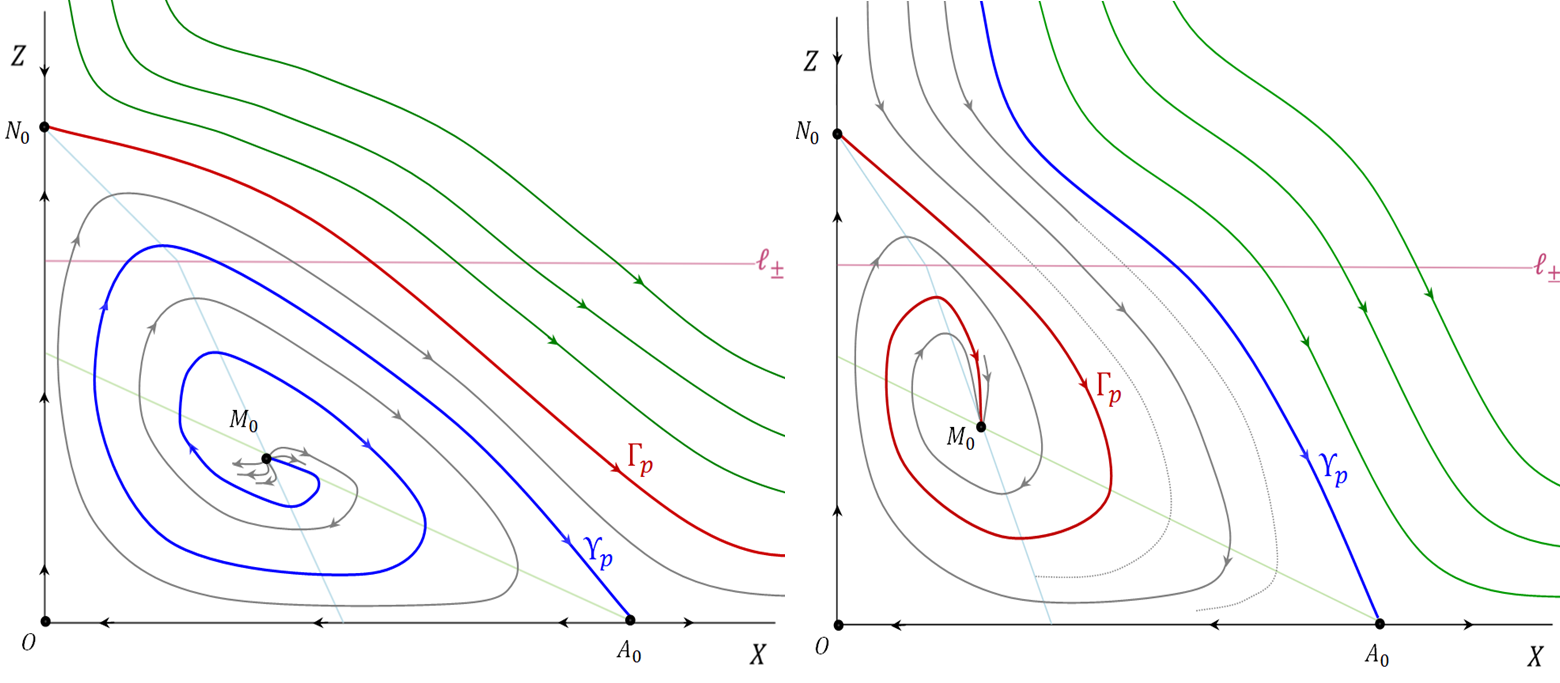}
	\caption{$p>p^{s,a}_\pm$, cases $p\in \C$ (LHS) and $p\in \Sl$ (RHS) without periodic orbits. For instance, this describes the range of $p$ where Dulac criterion holds, for both operators $\M^\pm$.}
	\label{Fig Dulac}
\end{figure}

\begin{figure}[!htb]
	\centering
	\includegraphics[scale=0.4]{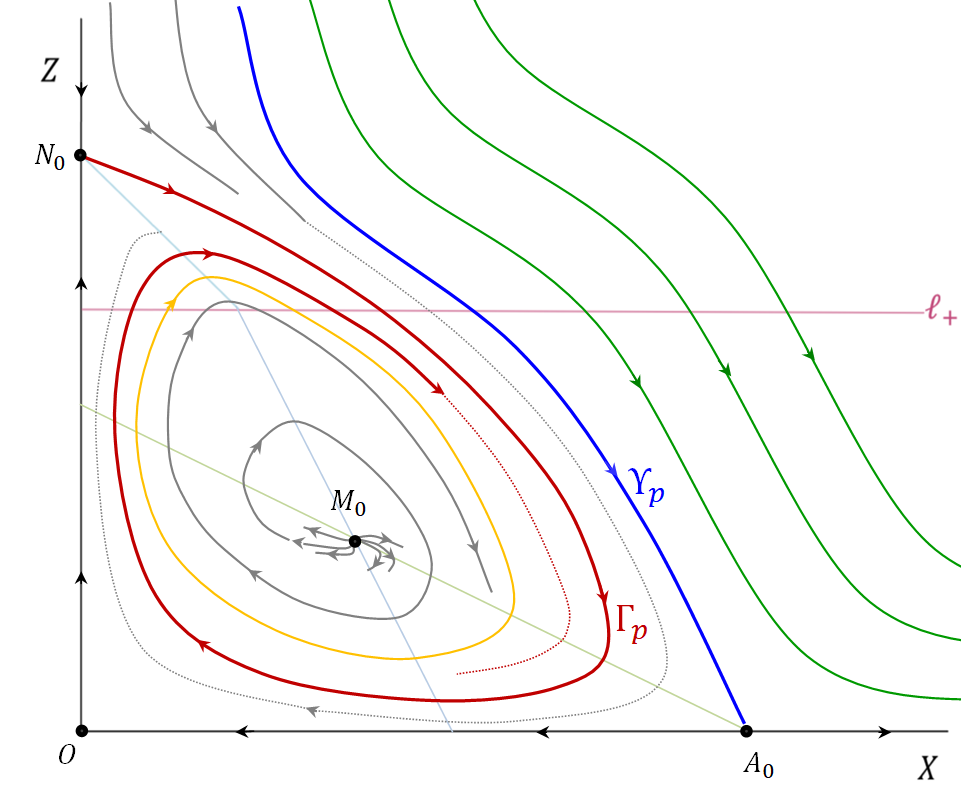}
	\caption{Case $p\in \Ps$ for $\M^+$, $p\in (p_{a+}^*,p^{p,a}_+)$. Here $M_0$ is a source. The orbits inside the displayed periodic orbit correspond to infinitely many $\alpha$--blowing up solutions with pseudo-slow decay at $+\infty$. All trajectories above $\Gamma_p$ correspond to solutions either in the exterior of a ball or in an annulus.}
	\label{Fig P M+}
\end{figure}

\begin{figure}[!htb]
	\centering
	\includegraphics[scale=0.4]{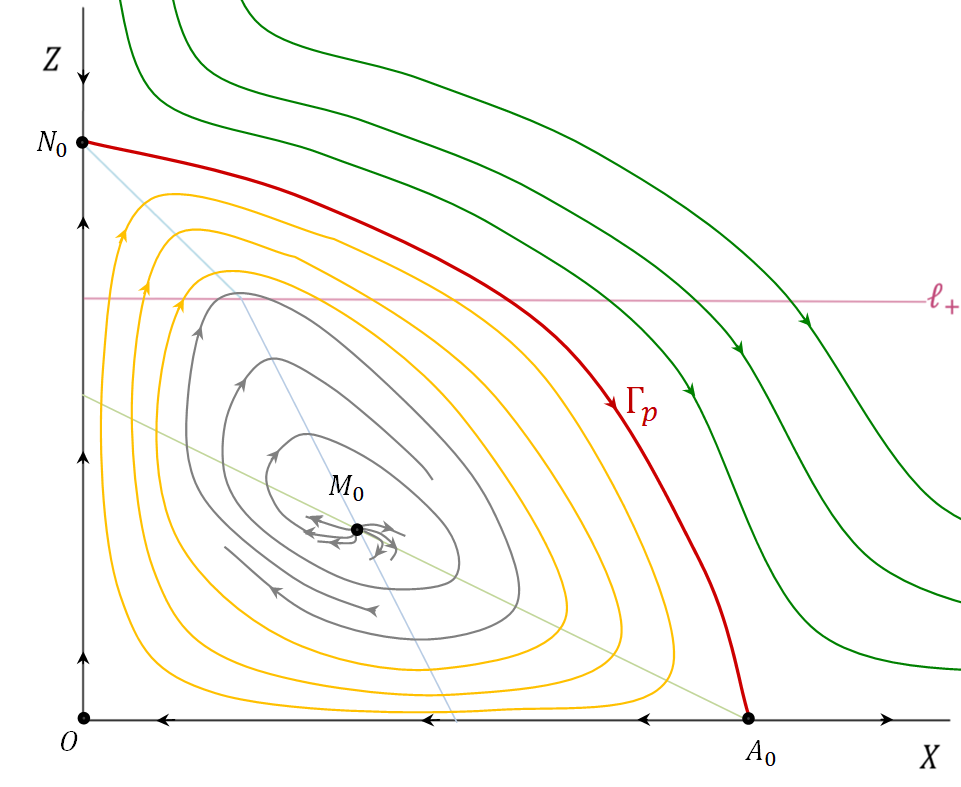}
	\caption{Case $p=p^*_{a+}\in \F$. Here $M_0$ is a source and $\Gamma_p=\Upsilon_p$, see Remark \ref{Gammap=Upsilonp}. There are infinitely many $\alpha$--blowing up solutions with pseudo-slow decay (inside the minimal periodic orbit), and infinitely many pseudo--blowing up solutions with pseudo-slow decay (periodic orbits).  Moreover, there are no solutions in the exterior of a ball.} 
	\label{Fig F M+}
\end{figure}

It is interesting to observe that when $p\leq p^{s,a}_+$ these results give a simple proof of some Liouville theorems in \cite{CutriLeoni}, concerning radial solutions. The same holds for $\M^-$, as we shall see in Section \ref{section Pucci-}.

\begin{rmk}\label{remark singular periodic}
For all $p>p^{s,a}_+$, as already mentioned in Section \ref{Introduction}, there exists a singular trivial solution given by $u_p=C_p\, r^{-\alpha}$, $C_p$ as in \eqref{omega M0}. This corresponds to the stationary trajectory $\Sigma_p\equiv M_0$. Moreover, any periodic orbit of the dynamical system \eqref{DS+} which intersects the concavity line $\ell_+$ twice corresponds to a classical pseudo--blowing up solution for the problem \eqref{P radial m}.	Instead, in the case $p=p^{p,a}_+$, periodic orbits around the center configuration of $M_0$ lying entirely in the region $R^-_\lambda$ (see Proposition \ref{prop M0 center}) cannot correspond to $C^2$ solutions, since they oscillate between the two functions $c_1r^{-\alpha}$ and $c_2r^{-\alpha}$ indefinitely for some $0<c_1<c_2$, without never changing convexity.
We stress that these two types of solutions originated from periodic orbits do exist in the case of the Laplacian in dimensions $N$ and $\tilde{N}_+$, for the critical exponents $p^a_\Delta$ and $p^{p,a}_+$ respectively, see Theorem 6.1(iii) in \cite{BV89arma}.
\end{rmk}

\begin{lem}\label{claim Gammap blows up M+}
	If $p> p^*_{a+}$ then $\Upsilon_p$ (see Proposition \ref{local uniqueness}) blows up in finite backward time. In particular, $\Upsilon_p$ does not correspond to a singular solution for any $p> p^*_{a+}$.
\end{lem}

\begin{proof}
If $p\in (p^*_{a+},p_+^{p,a}]$ we have $p\in \Ps$, and there exists a maximal periodic orbit $\theta_p$ around $M_0$ such that $\omega(\Gamma_p)=\theta_p$. If $\alpha(\Upsilon_p)=\theta_p$, then $\Gamma_p$ and $\Upsilon_p$ would cross somewhere. Indeed, this comes from the fact that the stable linear tangent direction at $A_0$ is above the line $\ell^+_2$ (see Figure \ref{Fig flow}), and the vector field on $\ell^+_2$ points down for $X>\alpha$, by Proposition \ref{prop flow} (3). Obviously crossings are not admissible by uniqueness of the ODE problem.
On the other side, for $p>p^{p,a}_+$, $M_0$ is a sink and periodic orbits are not allowed by Proposition \ref{Dulac}. 
Thus, in both cases, using Proposition \ref{AP bounds}, we get that $\Upsilon_p$ blows up backward in finite time.
\end{proof}

\begin{proof}[Proof of Theorem \ref{Th M+ singular}]
$(i)$ First we recall that for $p\in (1,p^{s,a}_+]$ there are no periodic orbits, by Proposition \ref{Dulac} and the proof of Proposition \ref{p=p^s}. 
Moreover, we know by Corollary \ref{cor separation p0} and Theorem~\ref{characterization p*} that the regular trajectory $\Gamma_p$ blows up forward in finite time. Hence, $\Gamma_p$ together with the line $L=\{(X,Z):X=\tilde{N}_+-2\}$ and the $X$, $Z$ axes, create a bounded region $D$ from which an orbit of \eqref{DS+} may only leave through $L$.

Thus, if $p<p^{s,a}_+$, any trajectory issued from $A_0$ (which is a source by Proposition \ref{local study stationary points}) crosses the line $L$ and then blows up in finite time. If $p=p^{s,a}_+$ the same holds, by Proposition \ref{p=p^s}.
In both cases there are infinitely many such trajectories which correspond to singular solutions in an interval $(0,R)$, $R>0$, see Proposition \ref{types of blow ups}.
They are $(\tilde{N}_+-2)$--blowing up, cf.\  \eqref{omega A0} with the $\omega$-limit exchanged by $\alpha$-limit.
Therefore there cannot be singular solutions in $\rN\setminus\{0\}$ for this range of $p$.

\vspace{0.07cm}

$(ii)$-$(iii)$ For $p\in (p^{s,a}_+, p^*_{a+})$, as in $(i)$, the trajectory $\Gamma_p$, the line $L$ and the $X,Z$ axes determine a bounded region $D$ from which an orbit may only leave through $L$. Recall that for these values of $p$, $A_0$ is a saddle point and $M_0$ is a source, see Proposition \ref{local study stationary points}. 
Thus, the unique orbit $\Upsilon_p$ arriving at $A_0$ (see Proposition \ref{local uniqueness}) can either converge to $M_0$ or to a periodic orbit around $M_0$, backward in time.
If $p\leq p^a_\Delta$ there are no periodic orbits (Proposition \ref{Dulac}), so $\Upsilon_{p}$ corresponds to $\alpha$--blowing up solution of \eqref{P radial m}; in particular this is the case for each $p\in (p^{s,a}_+,p^a_\Delta]$ if $p^{s,a}_+\le p_\Delta^a$. 

If in turn $p\in (p^a_\Delta,p^*_{a+})$ there could be periodic orbits around $M_0$, so that $\Upsilon_{p}$ corresponds to either a pseudo--blowing up or a $\alpha$--blowing up solution of \eqref{P radial m}.  

All the other orbits coming out from $M_0$ or from a periodic orbit $\theta$ around $M_0$ (whenever such $\theta$ exists) must necessarily leave $D$ by crossing the line $L$ in forward time, and therefore blow up in finite time. This gives infinitely many singular solutions of \eqref{P radial m}; they are either $\alpha$--blowing up or pseudo--blowing up in intervals $(0,R)$, $R>0$.

If $\theta$ exists, we have in addition infinitely many orbits $\tau$ issued from $M_0$ and converging to a minimal periodic orbit (which is $\theta$ if the system has only one limit cycle). Each $\tau$ crosses infinitely many times the line $\ell_+$ when $t\to +\infty$, and so corresponds to a $\alpha$--blowing up solution of \eqref{P radial m} with pseudo-slow decay at $+\infty$ as in \eqref{omega theta}. 
On the other hand, a periodic orbit itself in this range of $p$'s crosses $\ell_+$ twice, so corresponds to a pseudo--blowing up solution to \eqref{P radial m}, see Remark \ref{remark singular periodic}; they are pseudo-slow decaying and change concavity infinitely many times both as $r\to 0$ and $r\to +\infty$.

\vspace{0.07cm}

$(iv)$ When $p=p^*_{a+}$, the regular trajectory $\Gamma_{p^*_{a+}}$ together with the $X$ and $Z$ axes delimit an invariant set $D$ containing $M_0$. 
Since $M_0$ is a source, we have already seen that there exists a periodic orbit around $M_0$; say $\theta$ is the minimal one.
We then infer that there exist infinitely many periodic orbits in the region $D\setminus \mathrm{int}(\theta)$, at least in a neighborhood of $\partial D$, see Fig.~\ref{Fig F M+}. Indeed, the existence of a maximal periodic orbit $\theta_0$ inside $D$ would create a bounded region $D\setminus \mathrm{int}(\theta_0)$ in which the orbits issued from $\theta_0$ could not go anywhere, thus violating Poincaré-Bendixson theorem. 

\vspace{0.07cm}

$(v)$ When $p\in (p^*_{a+},p_+^{p,a})$, we have that $M_0$ is a source and there exists a minimal periodic orbit $\theta$ around $M_0$. 
In this case, $\theta$ crosses the line $\ell_+$ twice since $p<p^{p,a}_+$, see Proposition \ref{Dulac}. 
Thus, all trajectories issued from $M_0$ converge to  $\theta$ in forward time.
These and the periodic orbits give us singular solutions as in the last part of the proof of $(iii)$. Finally, note that no singular solutions converge to $A_0$ by Lemma  \ref{claim Gammap blows up M+}, so the assertion holds.

\vspace{0.07cm}
	
$(vi)$ By Propositions \ref{local study stationary points}, \ref{Dulac}, and \ref{prop M0 center} we have that for $p=p^{p,a}_+$ the stationary point $M_0$ is a center while for $p>p^{p,a}_+$ $M_0$ is a sink without periodic orbits.
These and the fact that $A_0$ is a saddle point, whose unstable manifold is the $X$ axis, imply that no singular nontrivial solutions are admissible.
\end{proof}

Finally we consider the case of exterior domain solutions, proving Theorem \ref{Th exterior ball+-F Introd} for $\M^+$. The proof for $\M^-$ turns out to be the same.

In Section \ref{subsection exterior} we have observed that a solution $u$ of \eqref{shooting derivative} necessarily satisfies \eqref{mu delta}. Hence the corresponding trajectory $\varXi_p$ blows up backward in finite time, see Proposition \ref{prop exterior}. Thus, to prove Theorem \ref{Th exterior ball+-F Introd} it is enough to show that for $p\in (1,p^*_{a+}]$ there are no orbits of the dynamical system \eqref{DS+}  defined in $(T,+\infty)$ for some $T>0$
with this kind of blow-up behavior.

\begin{proof}[Proof of Theorem \ref{Th exterior ball+-F Introd}]
By the definition and properties of the critical exponent in Sections \ref{section crossing}, \ref{section critical exponent}, we know that $p\in \C$ for $p\in (1,p^*_{a+})$, while $p^*_{a+}\in \F$.  
In the first case the regular trajectory $\Gamma_p$ together with the $X$ and $Z$ axes and the line $L=\{(X,Z):X=\tilde{N}_+-2\}$ bound a region $D$ from which any trajectory can only escape in forward time through $L$. In the second case $\Gamma_p$ and the $X$ and $Z$ axes enclose a bounded invariant region $D$.
In both cases the closure of $D$ contains the points $M_0$ (for $p\geq p^{s,a}_+$) and $A_0$. 

By contradiction assume that a radial solution of \eqref{ProbExterior} exists. Then, by Proposition \ref{prop exterior} the corresponding trajectory $\varXi_p$ is defined in an interval $(t_\delta,+\infty]$ for some $t_\delta>-\infty$, and blows at $t_\delta$ satisfying \eqref{blow up exterior 1Q}.
Since it does not blow up in forward time, by Proposition \ref{AP bounds} (see \eqref{X< tildeN-2}) and Poincaré-Bendixson theorem the $\omega$-limit $\omega(\varXi_p)$ is either $M_0$ (if $p>p^{s,a}_+$), or $A_0$, or a periodic orbit around $M_0$.
In any case $\varXi_p$ should cross $\Gamma_p$ which is not possible.
\end{proof}
	
\begin{rmk}\label{remark exterior}
It is proved in \cite{GILexterior2019}, when $a=0$, that for every $p\in (p^*_{a+},+\infty)$ both a fast decaying and infinitely many slow or pseudo-slow decaying solutions of \eqref{ProbExterior} exist.
In terms of our quadratic system \eqref{DS+} this could be proved using Lemma \ref{claim Gammap blows up M+} for the fast decaying solutions, or studying the trajectories arriving at $M_0$ or at a periodic orbit for the slow or pseudo-slow decaying solutions. However, since the proof of \cite{GILexterior2019} easily extends to the case $a\neq 0$, we prefer to omit the details. 
\end{rmk}

Note that with the analysis of the trajectories blowing up backward in finite time one can only get the existence of a solution $u$ satisfying \eqref{mu delta} at some radius $\mu>0$. Then the solution should be continued (in $3Q$) to reach a positive radius $\rho_0>1$ where $u(\rho_0)=0$, so to verify the Dirichlet problem in the exterior of a ball. This is possible by using a shooting argument from $\mu$, as done for instance in \cite[proof of Theorem 6.1]{GILexterior2019}.

\section{The $\mathcal{M}^-$ case}\label{section Pucci-}

In this section we analyze the complementary case for the operator $\mathcal{M}^-$. Recall its respective dimension-like parameter from \eqref{def dimensional-like} satisfying $\tilde{N}_- \geq N$.
The main difference with the case of $\M^+$ is the reverse ordering of the exponents $p^a_\Delta$ and $p^{p,a}_-$ from \eqref{critical exponents a}, with $p^{s,a}_-\leq p^{p,a}_-\leq p_\Delta^a$. Also, if $\lambda<\Lambda$ then the stationary point $M_0$ is a sink in the interval $(p^{p,a}_-,p^a_\Delta)$, see Proposition \ref{local study stationary points} (4).

\smallskip

We start by pointing out that all properties stated for $\M^+$ in section \ref{section crossing} also hold for $\M^-$. In particular, one gets that the set $\F$ possesses at most one point, which splits the interval $(1,+\infty)$ into two components $\C$ and $\Ps\cup\Sl$ accordingly to Corollary \ref{cor separation p0}.
Moreover, each of these components is nonempty, since one verifies, as in Propositions \ref{p>p^p+} and \ref{p<p_Delta}, the following result.

\begin{prop}\label{p>p_Delta^a and p<p^{p,a}-}
If $p> p_\Delta^a$ then $p\in \mathcal{S}$, and for $p<p_-^{p,a}$ it holds that $p\in \mathcal{C}$.
\end{prop}

This allows to define, as in Section \ref{section Pucci+}, the critical exponent $p^*_{a-}$ as follows
\begin{center}
$p^*_{a-}=\,\sup \C$.
\end{center}
Then, by Proposition \ref{p>p_Delta^a and p<p^{p,a}-}, $p^*_{a-}\in [\,p^{p,a}_-,\, p^a_\Delta\,]$ and, as for $\M^+$, we call it the critical exponent for $\M^-$.
Next we show that $p^*_{a-}\in \F$ and $p^*_{a-}$ is in the interior of the previous interval.

\begin{theorem}\label{th 5.2F}
The critical exponent $p^*_{a-}$ belongs to $\F$. Thus it is the only exponent in the equation \eqref{P radial m-} for which there exists a unique, up to scaling, fast decaying solution.

\vspace{0.03cm}

Moreover, if $\lambda<\Lambda$, then \eqref{critical ordering M-} holds and there exists $\varepsilon>0$ such that $(p^a_\Delta-\varepsilon,+\infty)\subset \Sl$.
\end{theorem}

\begin{proof}
Obviously $p^*_{a-}\not\in \C$ because $\C$ is open, see Remark \ref{C is open}. Moreover, $p^*_{a-}$ cannot belong to $\Ps$; otherwise $\Gamma_{p^*_{a-}}$ should cross the line $\ell_1^+$ by Lemma \ref{lem l1, l2 for X>alpha}(ii), while $\Gamma_p$ for $p\in \C$ never does it.

Finally we show that $p^*_{a-}\not \in \Sl$. Indeed, if this was the case then $p^*_{a-}>p^{p,a}_-$ because $M_0$ is a center at $p^{p,a}_-$, see Proposition \ref{prop M0 center}.
Hence $M_0=M_0(p)$ is a sink for every $p$ in a neighborhood $I_\varepsilon=(p^*_{a-}-\varepsilon, \, p^*_{a-}+\varepsilon)$ for some $\varepsilon>0$, by Proposition \ref{local study stationary points} (4). 
Then there exists a maximal ball $B_{\eta_p}$ centered at $M_0(p)$ with the property that any trajectory $\tau_p$ entering in $B_{\eta_p}$ satisfies $\omega(\tau_p)=M_0(p)$, see \cite{Hale}.	
Since we are assuming that $p^*_{a-}\in \Sl$ then $\omega(\Gamma_{p^*_{a-}})=M_0(p^*_{a-})$. By the continuity of the dynamical system with respect to the parameter $p$, also $\omega(\Gamma_p)=M_0(p)$ for $p\in I_\varepsilon$ (up to diminishing $\varepsilon$). 
But this contradicts the definition of $p^*_{a-}$, since $\Gamma_q$ blows up in finite time when $q\in \C$. 

\vspace{0.03cm}

Hence, $p^*_{a-}\in \F$. The proof that $p^*_{a+}$ cannot be $p^{p,a}_-$ nor $p^a_\Delta$ is the same as the one for $\M^+$, see Theorem \ref{characterization p*}. It relies on Proposition \ref{Dulac} which states, in particular, that there are no periodic orbits of \eqref{DS-} for $p^a_\Delta$. This also proves that the regular trajectory $\Gamma_{p^a_\Delta}$ converges to $M_0=M_0(p^a_\Delta)$, so that $p^a_\Delta\in \Sl$.
Next, a continuity argument as in the first part of this proof shows that $p\in (p^a_\Delta-\varepsilon,p^a_\Delta)$ also belongs to $\Sl$ for sufficiently small $\varepsilon>0$.
Consequently, for such $p$'s it does not exist a periodic orbit around $M_0$, and the proof is complete.
\end{proof}

\begin{proof}[Proof of Theorem \ref{Th M-}]
All previous results obtained for $\M^-$ prove the theorem. In particular, the statements $(iii)$--$(iv)$ follow from Theorem \ref{th 5.2F}.
\end{proof}

We finish the section with the proof of Theorem \ref{Th singular- Introd} about singular solutions. 

\begin{proof}[Proof of Theorem \ref{Th singular- Introd}]
It is enough to prove $(iii)$ and $(v)$, since the proof of the other items are the same as for Theorem \ref{Th M+ singular}.  Let us analyze the whole interval $p\in(p^{p,a}_-,p_\Delta^a]$.
 
Recall that $M_0$ is a sink whenever $p> p^{p,a}_-$. 
Therefore, there is no trajectory coming out from $M_0$ in the range $p\in (p_-^{p,a},p_\Delta^a]$. 
In particular, $\alpha(\Upsilon_p)$ is never $M_0$ in this range of $p$.

For $p\in (p^{p,a}_-,p^*_{a-})$ we have $p\in \mathcal{C}$. In this case the regular trajectory $\Gamma_p$ and the line $L=\{(X,Z): X=\tilde{N}_--2\}$, together with the $X$ and $Z$ axes, create a bounded region from which any orbit may only leave forward in time through $L$; recall that the flow is going out on $L$, see Figure \ref{Fig flow}.
Therefore, Poincaré-Bendixson theorem implies the existence of a periodic orbit $\theta_p$ around $M_0$ such that $\alpha(\Upsilon_p)=\theta_p$. 
This immediately determines four types of nontrivial positive pseudo--blowing up solutions of \eqref{P}--\eqref{H singular} (in the case of $\M^-$):

(1) a fast decaying solution corresponding to the trajectory $\Upsilon_p$;

(2) solutions with slow decay, whose corresponding orbits lie inside a minimal periodic orbit $\theta_0$ around $M_0$; here $\theta_0$ crosses $\ell_-$ twice due to Proposition \ref{Dulac};

(3) solutions of the Dirichlet problem in $B_R\setminus\{0\}$, such that the corresponding orbits are issued from $\theta_p$ and blow up in finite forward time;

(4) pseudo-slow decaying solutions, which correspond to the periodic orbits.
\\
All of these singular solutions change concavity infinitely many times in a neighborhood of $r=0$. Further, there are infinitely many solutions of types (2) and (3); see Figure \ref{Fig Psingular M-}. 
Thus, $(iii)$ holds. 

To prove $(v)$, let us recall that at $p_\Delta^a$ no periodic orbits are admissible by Proposition \ref{Dulac}. Also, by Theorem \ref{th 5.2F} there exists $\varepsilon>0$ such that $p\in \Sl$ for all $p\in (p_\Delta-\varepsilon,+\infty)$.
Now, arguing as in Lemma \ref{claim Gammap blows up M+} one sees that $\Upsilon_p$ blows up in finite backward time for $p>p^*_{a-}$, so $(v)$ is proved.
\end{proof}

\begin{rmk}
In the case of $\M^-$ the existence of singular solutions in the range $(p^*_{a-},p^a_\Delta-\varepsilon)$ is not guaranteed, though solutions as in the cases $(2)$--$(4)$ in the proof above are admissible.
\end{rmk}

Concerning the exterior domain solutions, we have already observed in Section \ref{section singular} that the proof of Theorem \ref{Th exterior ball+-F Introd} is the same for both operators $\M^\pm$.

\begin{figure}[!htb]
	\centering
	\includegraphics[scale=0.4]{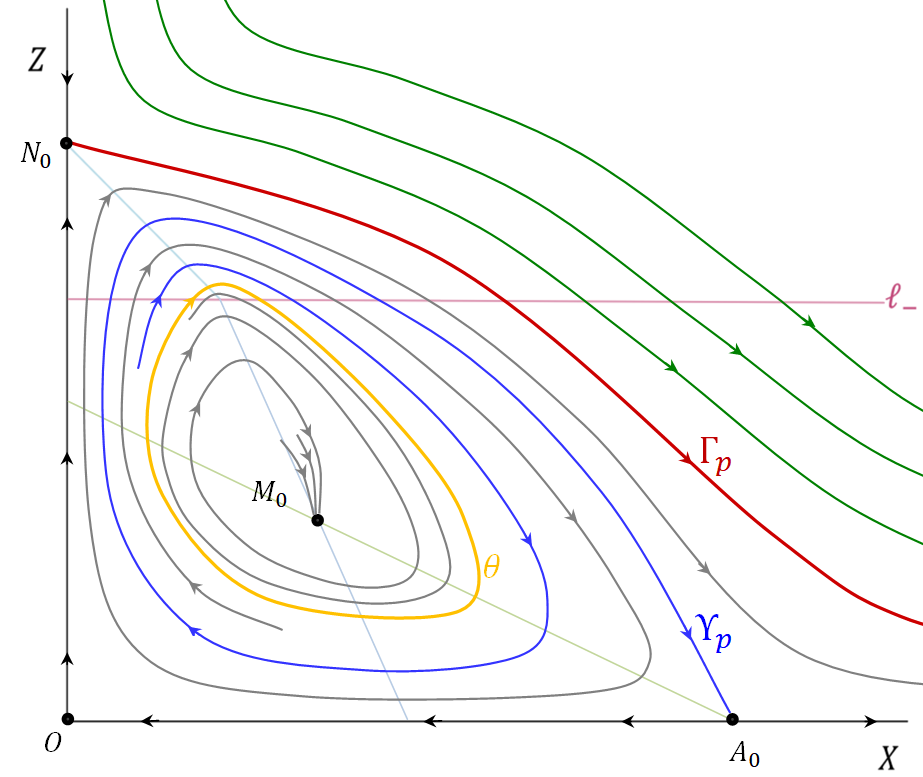}
	\caption{Case $p\in (p^{p,a}_-,p^*_{a-})$ for $\M^-$, $M_0$ is a sink. There are infinitely many pseudo--blowing up solutions: a unique fast decaying (given via the orbit $\Upsilon_p$); a pseudo-slow decaying (periodic orbit~$\theta$); infinitely many in a ball (outside $\theta$); infinitely many slow decaying (inside $\theta$).} 
	\label{Fig Psingular M-}
\end{figure}

\section{Appendix}

\subsection{Local study}

In this section we detail the proof of Proposition \ref{local study stationary points}. The linearization for $\M^+$, $\tilde{N}=\tilde{N}_+$, is
\begin{align*}
L (X,Z) =
\left(
\begin{array}{cc}
\partial_X f^+ & \partial_Z f^+ \\
\partial_X g^+ & \partial_Z g^+
\end{array}
\right)
= \left(
\begin{array}{cc}
2X - (N-2) + \frac{Z}{\lambda} & \frac{X}{\lambda} \\
-pZ & N+a-pX-\frac{2Z}{\lambda}
\end{array}
\right) \quad \textrm{in $R^+_\lambda$,}
\end{align*}
\begin{align*}
L (X,Z)=
\left(
\begin{array}{cc}
\partial_X f^- & \partial_Z f^- \\
\partial_X g^- & \partial_Z g^-
\end{array}
\right)=
\left(
\begin{array}{cc}
2X - (\tilde{N}-2) + \frac{Z}{\Lambda} & \frac{X}{\Lambda} \\
-pZ & \tilde{N}+a-pX-\frac{2Z}{\Lambda}
\end{array}
\right)  \quad \textrm{in $R^-_\lambda$.}
\end{align*}
For instance, at $N_0=(0,\lambda (N+a))$ and $A_0=(\tilde{N}-2,0)$ one has
\begin{center} 
	$L(N_0) = \left(\begin{array}{cc}2+a & 0 \\-p\lambda (N+a) & -N-a\end{array}\right),
	$
\;
	$L(A_0) = \left(\begin{array}{cc}\tilde{N} - 2 & \frac{\tilde{N} - 2}{\Lambda} \\0 & \tilde{N} +a-p(\tilde{N} - 2)\end{array}\right).$
\end{center}
The eigenvalues for $N_0$ are $2$ and $-N-a$, while for $A_0$ are $\sigma_1 = \tilde{N}_\pm - 2$ and $\sigma_2 = \tilde{N}_\pm +a- p(\tilde{N} - 2) $.
Recall that $M_0=(X_0,Z_0)$, where $X_0 = \alpha=\frac{a+1}{p-1}$ and $Z_0 =\Lambda(\tilde{N}-p\alpha+a) = \Lambda(\tilde{N}-2 - \alpha)$,
\begin{align*}
L(M_0) = \left(
\begin{array}{cc}
\alpha & \frac{\alpha}{\Lambda} \\
-p\Lambda(\tilde{N} - p\alpha+a) & -(\tilde{N} - p\alpha+a)
\end{array}
\right).
\end{align*}
In order to analyze the eigenvalues of $L(M_0) $ one needs to look at the roots of the equation
\begin{center}
	$\sigma^2 + \sigma\left( \frac{Z_0}{\Lambda} - X_0 \right) + X_0 (p-1) \frac{Z_0}{\Lambda}  = 0.
	$
\end{center}
They are given by
$2\sigma_\pm = { X_0 - \frac{Z_0}{\Lambda} \pm \sqrt{\Delta} }$, 
where $\Delta = \left( \frac{Z_0}{\Lambda} - X_0 \right)^2 - 4(2+a) \frac{Z_0}{\Lambda}$.

Note that $X_0 = \frac{Z_0}{\Lambda}$ is equivalent to
$ \alpha = \frac{\tilde{N}-2}{2}$, i.e.\ $p=p^{p,a}_+$.
In this case $\mathrm{Re}(\sigma_\pm)=0$ and the roots are purely imaginary.
Moreover, $X_0 > \frac{Z_0}{\Lambda}\Leftrightarrow  p<p^{p,a}_+$, and $X_0 < \frac{Z_0}{\Lambda}\Leftrightarrow  p>p^{p,a}_+$.

If $\textrm{Im}(\sigma_\pm)\neq 0$, this already determines the sign of $\mathrm{Re}(\sigma_\pm)$.
Assume then $\textrm{Im}(\sigma_\pm)= 0$ i.e.\  $\sigma_\pm\in \real$.
Observe that $\Delta<(X_0-\frac{Z_0}{\Lambda} )^2 $ as far as $M_0$ stays in $1Q$. This yields $\sigma_->0$ for $p_+^{s,a}<p<p_+^{p,a}$ (so $\sigma_\pm>0$ and $M_0$ is a source); while $\sigma_+<0$ if $p>p_+^{p,a}$ (so $\sigma_\pm<0$ and $M_0$ is a sink).

It is possible to prove that $M_0$ is a saddle point in the fourth quadrant when $1<p<p^{s,a}_+$. However, this would correspond to solutions of the absorption problem $\mathcal{M}^+u-u^p=0$. See \cite{BV} in the case of the Laplacian operator.

\subsection{Energy analysis}

Let us consider the energy functional $E$ of the operator $\M^+$ in the region $R^-_\lambda$, which is a slight variation of the energy of the Laplacian operator in dimension $\tilde{N}_+$ treated in \cite{BV},
\begin{align*}
\textstyle{E(t,X,Z)
=e^{t(\tilde{N}_+-2-2\alpha)} \,X (XZ)^\alpha \left\{ \frac{X}{2}+\frac{Z}{\Lambda (p+1)}-\frac{\tilde{N}_+}{p+1}
\right\} \quad \textrm{ in } R^-_\lambda\cup \ell_+ }
\end{align*}
understood as natural extension up to $\ell_+$.
In terms of $u$, the energy functional $E$ for $\M^+$ reads as
\begin{align*}
\textstyle{E(r)=E(r,u)=
r^{\tilde{N}_+} \left( \,\frac{(u^\prime)^2}{2} +\frac{1}{\Lambda}\frac{r^au^{p+1}}{p+1} \right) +\frac{\tilde{N}_+}{p+1} \,u u^\prime  r^{\tilde{N}_+-1} \quad \textrm{ if } u^{\prime\prime}\geq 0.}
\end{align*}
Of course these two expressions are equivalent after the transformation \eqref{X,Z a}. Moreover,
\begin{center}
$E^\prime(r)=
\,{r^{\tilde{N}_++a-1}} (u^\prime)^2
\left( \frac{\tilde{N}_++a}{p+1}-\frac{\tilde{N}_+-2}{2}
\right) \quad \textrm{ if } u^{\prime\prime}\geq 0 ,$
\end{center}
and so the following monotonicity holds
\begin{align}\label{sign Eprime}
\dot{E}<0 \; \textrm{ if } p>p_+^{p,a}\, , \quad 
\dot{E}=0 \; \textrm{ if } p=p_+^{p,a}\, , \quad
\dot{E}>0 \; \textrm{ if } p<p_+^{p,a}
\qquad \textrm{ in } \;R^-_\lambda\cup \ell_+.
\end{align}

Now we investigate the precise behavior of the trajectories close to $M_0$ at $p=p^p_\pm$. 
Here, $\lambda \leq \Lambda$ and the result gives an alternative proof in the case of the Laplacian operator $\lambda=\Lambda=1$ in \cite{BV89arma}.

\begin{prop}\label{prop M0 center}
	$M_0$ is a center when $p=p^{p,a}_\pm$. 
\end{prop}

\begin{proof} We present the proof for $\M^+$; for $\M^-$ it is the same in light of Section \ref{section Pucci-}. Let $\tau=(X,Z)$ be an orbit contained in $R^-_\lambda\cup \ell_+$. 
Let us show that $\tau$ is periodic.
To simplify notation let $a=0$, $\tilde{N}=\tilde{N}_+$. The energy of $\tau$ on the line $\ell_2^+$ is given by
	\begin{center}
		$E_{|_{\ell_2^+ \cap R^-_\lambda}}=E^-_{|_{\ell_2^+}}(X)  = - \frac{\Lambda^\alpha}{\tilde{N}} \,X^{\alpha+2} \,  (\tilde{N}-pX)^\alpha  \quad 
		\textrm{ at }\, p=p^p_+ =p^{p,0}_+.$
	\end{center}
	Since the energy is a constant function of $t$ when $p=p^p_+ $ with $\frac{\alpha+2}{\alpha}=p$, then
	\begin{align}\label{eq h(X) l2}
	(\tilde{N}-pX)X^p \equiv c>0 \quad \textrm{on}\;\;\ell_2^+.
	\end{align}
	
	Now we may translate the information from \eqref{eq h(X) l2} in terms of the function $h$ defined as
	\begin{center}
		$h(X)=(\tilde{N}-pX)X^p$,\quad  for $X\in [\,{1}/{p},\,{\tilde{N}}/{p}\,]$,\;\; where $p=p^p_+$,
	\end{center}
	for which \eqref{eq h(X) l2} represents its level curves. The domain $[\,{1}/{p},\,{\tilde{N}}/{p}\,]$ entails the behavior of $h$ in the respective interval delimited by $\ell_2^+$ on $R^-_\lambda$, up to the boundary.
	
	Let us analyze the function $h$; it is positive at $1/p$, and equals to zero at $\tilde{N}/p$.
	Since 
	\begin{center}
		$h^\prime (X)=pX^p (\,\frac{\tilde{N}}{X} -1-p)$, \quad with \, $\frac{\tilde{N}}{1+p}=\frac{\tilde{N}-2}{2}=\alpha$, \; $p=p^p_+$,
	\end{center}
	then $h$ is increasing when $X<\alpha$, decreasing for $X>\alpha$, and it assumes the positive maximum value $h(\alpha)=\alpha^{p+1}:=c_{\infty}$ at $X=\alpha$.
	Moreover, note that $h(1/p)=(\tilde{N}-1)(1/p)^p+ :=c_1$.
	
	\smallskip
	
	Here $h$ is a polynomial function which prescribes the value of the energy on $\ell_2^+\cap R^-_\lambda$, namely $h=-E^{1/\alpha}\equiv c$.
	For any $k\in \n$ with $c_k\in [\,c_1,c_{\infty}\,)$, the line $h\equiv c_k$ intersects the graph of $h$ at exactly two points $X_1^k,X_2^k$ such that $X_1^k<\alpha<X_2^k$. 
	Also, they satisfy
	\begin{align}\label{eq center ck}
	\textrm{$c_k = h(X_1^k)=h(X_2^k)\rightarrow  h(\alpha)=c_\infty$ \;\; when\;\; $X_i^k \rightarrow \alpha$\;\, as $k\rightarrow +\infty$,\, $i=1,2$.}
	\end{align}
	Furthermore, the line $h=c_\infty$ intersects the graph of $h$ only once at the point $X=\alpha$.
	
	In our phase plane context, this means that any trajectory $\tau$ contained in $R^-_\lambda\cup\ell_+$ bisects the line $\ell_2^+$ at exactly two points $P_1=(X_1,Z_1)$, $P_2=(X_2,Z_2)$, with $X_1< \alpha$, $X_2 > \alpha$.
	By Proposition~\ref{prop flow} (3) the flow moves horizontally on $\ell_2^+$, namely to the right for $X<\alpha$, and to the let when $X> \alpha$.
	
Observe that $\ell_2^+$ is a transversal section to the flow, on which any trajectory approaching $M_0$ must pass across, either in the past or in the future.	Hence, the trajectory $\tau$ has to be closed, by moving clockwisely. 	Since this dynamics is realized for any trajectory contained on $R^-_\lambda\cup\ell_+$, and \eqref{eq center ck} holds, in particular any trajectory close to $M_0$ is periodic, so $M_0$ is a center. 
\end{proof}

\smallskip

\textbf{{Acknowledgments.}} 
Part of this work was done when the first and second authors were visiting Sapienza University of Rome. They would like to thank the warm hospitality and care provided by the people in the Math Department.

\smallskip

L.\ Maia was supported by FAPDF, CAPES, and CNPq grant 308378/2017 -2. G.\ Nornberg was supported by FAPESP grants 2018/04000-9 and 2019/031019-9, São Paulo Research Foundation. 
F.\ Pacella was supported by INDAM-GNAMPA.

\bibliography{bibtex}

\begin{thebibliography}{10}

\bibitem{BApisa}
S.~N. Armstrong and B.~Sirakov.
\newblock Sharp {L}iouville results for fully nonlinear equations with
  power-growth nonlinearities.
\newblock {\em Ann. Sc. Norm. Super. Pisa Cl. Sci. (5)}, 10(3):711--728, 2011.

\bibitem{BAScpam}
S.~N. Armstrong, B.~Sirakov, and C.~K. Smart.
\newblock Fundamental solutions of homogeneous fully nonlinear elliptic
  equations.
\newblock {\em Comm. Pure Appl. Math.}, 64(6):737--777, 2011.

\bibitem{BV89arma}
M.-F. Bidaut-V{\'e}ron.
\newblock Local and global behavior of solutions of quasilinear equations of
  {E}mden-{F}owler type.
\newblock {\em Archive for Rational Mechanics and Analysis}, 107(4):293--324,
  1989.

\bibitem{BV}
M.~F. Bidaut-Veron and H.~Giacomini.
\newblock A new dynamical approach of {E}mden-{F}owler equations and systems.
\newblock {\em Adv. Differential Equations}, 15(11-12):1033--1082, 2010.

\bibitem{BGLPcv}
I.~Birindelli, G.~Galise, F.~Leoni, and F.~Pacella.
\newblock Concentration and energy invariance for a class of fully nonlinear
  elliptic equations.
\newblock {\em Calc. Var. Partial Differential Equations}, 57(6):Art. 158,
  1--22, 2018.

\bibitem{CafCab}
L.~A. Caffarelli and X.~Cabr\'{e}.
\newblock {\em Fully nonlinear elliptic equations}, volume~43 of {\em American
  Mathematical Society Colloquium Publications}.
\newblock American Mathematical Society, Providence, RI, 1995.

\bibitem{CGS}
L.~A. Caffarelli, B.~Gidas, and J.~Spruck.
\newblock Asymptotic symmetry and local behavior of semilinear elliptic
  equations with critical {S}obolev growth.
\newblock {\em Comm. Pure Appl. Math.}, 42(3):271--297, 1989.

\bibitem{ChiconeTian}
C.~Chicone and J.~H. Tian.
\newblock On general properties of quadratic systems.
\newblock {\em Amer. Math. Monthly}, 89(3):167--178, 1982.

\bibitem{DjairoCM}
P.~Cl{\'e}ment, D.~G. de~Figueiredo, and E.~Mitidieri.
\newblock Quasilinear elliptic equations with critical exponents.
\newblock {\em Topological Methods in Nonlinear Analysis}, 7(1):133--170, 1996.

\bibitem{CutriLeoni}
A.~Cutr{\`\i} and F.~Leoni.
\newblock On the {L}iouville property for fully nonlinear equations.
\newblock {\em Annales de l'Institut Henri Poincare, {S}ection (C)},
  17(2):219--245, 2000.

\bibitem{BdaLioSym}
F.~Da~Lio and B.~Sirakov.
\newblock Symmetry results for viscosity solutions of fully nonlinear uniformly
  elliptic equations.
\newblock {\em J. Eur. Math. Soc. (JEMS)}, 9(2):317--330, 2007.

\bibitem{FQaihp}
P.~L. Felmer and A.~Quaas.
\newblock On critical exponents for the {P}ucci's extremal operators.
\newblock {\em Ann. Inst. H. Poincar\'{e} Anal. Non Lin\'{e}aire},
  20(5):843--865, 2003.

\bibitem{FQind}
P.~L. Felmer and A.~Quaas.
\newblock Critical exponents for uniformly elliptic extremal operators.
\newblock {\em Indiana Univ. Math. J.}, 55(2):593--629, 2006.

\bibitem{GILexterior2019}
G.~Galise, A.~Iacopetti, and F.~Leoni.
\newblock Liouville-type results in exterior domains for radial solutions of
  fully nonlinear equations.
\newblock {\em To appear in Journal of Differential Equations}, 2019.

\bibitem{GLPradial}
G.~Galise, F.~Leoni, and F.~Pacella.
\newblock Existence results for fully nonlinear equations in radial domains.
\newblock {\em Comm. Partial Differential Equations}, 42(5):757--779, 2017.

\bibitem{GGN}
F.~Gladiali, M.~Grossi, and S.~L.~N. Neves.
\newblock Nonradial solutions for the {H}\'{e}non equation in {$\Bbb{R}^N$}.
\newblock {\em Adv. Math.}, 249:1--36, 2013.

\bibitem{GQsingular}
O.~Gonz\'{a}lez-Melendez and A.~Quaas.
\newblock On critical exponents for {L}ane-{E}mden-{F}owler-type equations with
  a singular extremal operator.
\newblock {\em Ann. Mat. Pura Appl. (4)}, 196(2):599--615, 2017.

\bibitem{Hale}
J.~K. Hale and H.~Ko\c{c}ak.
\newblock {\em Dynamics and bifurcations}, volume~3 of {\em Texts in Applied
  Mathematics}.
\newblock Springer-Verlag, New York, 1991.

\bibitem{BQnonproper}
A.~Quaas and B.~Sirakov.
\newblock Existence results for nonproper elliptic equations involving the
  {P}ucci operator.
\newblock {\em Comm. Partial Differential Equations}, 31(7-9):987--1003, 2006.

\end{thebibliography}
\bibliographystyle{abbrv}

\end{document}